\documentclass[12pt,reqno]{amsart}
\usepackage{amssymb,latexsym}
\usepackage{amsmath}
\usepackage{amsthm}
\usepackage{graphicx}
\usepackage{hyperref}
\usepackage{titletoc}
\usepackage{pdfsync}
\usepackage{color,xcolor}
\usepackage{amsthm,amsmath,amssymb}

\usepackage{mathrsfs}

\numberwithin{equation}{section}
\newtheorem{theorem}{Theorem}[section]
\newtheorem{proposition}[theorem]{Proposition}
\newtheorem{lemma}[theorem]{Lemma}

\usepackage{multirow}
\usepackage[font=bf,aboveskip=15pt]{caption}
\usepackage[toc,page]{appendix}

\theoremstyle{definition}

\newtheorem{lem}[theorem]{Lemma}

\newcommand{\inn}{{\quad\hbox{in } }}

\def\R{{\mathfrak R}}

\def\begeq{\begin{equation}}
	\def\endeq{\end{equation}}

\def\R{\Bbb R}

\allowdisplaybreaks

\begin{document}
	
	\title[prescribed curvature problem of polyharmonic operators]
 {Double-tower solutions for higher order prescribed curvature problem}
	
	\author{Yuan Gao, Yuxia Guo AND Yichen Hu}

	\address{Yuan Gao,
		\newline\indent Department of Mathematical Sciences, Tsinghua University,
		\newline\indent Beijing 100084,  P. R. China.
	}
	\email{gaoy22@mails.tsinghua.edu.cn}
	
	\address{Yuxia Guo,
		\newline\indent Department of Mathematical Sciences, Tsinghua University,
		\newline\indent Beijing 100084,  P. R. China.
	}
	\email{yguo@tsinghua.edu.cn}

	\address{Yichen Hu,
		\newline\indent Department of Mathematical Sciences, Tsinghua University,
		\newline\indent Beijing 100084,  P. R. China.}
	\email{hu-yc19@mails.tsinghua.edu.cn}
\thanks{This work  is supported by NSFC (No. 12031015, 12271283).}
	\begin{abstract}
		We consider the following higher order prescribed  curvature
		problem  on  $ {\mathbb{S}}^N : $
		\begin{equation*}
			D^m \tilde u=\widetilde{K}(y) \tilde u^{m^{*}-1} \quad \mbox{on} \ {\mathbb {S}}^N, \qquad \tilde u >0 \quad \inn{\mathbb {S}}^N.
		\end{equation*}
		where $\widetilde{K}(y)>0$ is a  radial function, $m^{*}=\frac{2N}{N-2m}$ and $D^m$ is $2m$ order differential operator given by
  \begin{equation*}
      D^m=\prod_{i=1}^m\left(-\Delta_g+\frac{1}{4}(N-2i)(N+2i-2)\right),
  \end{equation*}
  where $g=g_{{\mathbb{S}}^N}$is the Riemannian metric. We prove the existence of infinitely many  double-tower type solutions, which  are invariant under some non-trivial sub-groups of $O(3),$ and their energy can be made arbitrarily large. 
		
		\vspace{2mm}
		
		{\textbf{Keyword:}  Prescribed curvature equations, Higher order operators,  Double-tower solutions}.
		
		\vspace{2mm}
		
		{\textbf{AMS Subject Classification:}
			35A01, 35G20, 35J91.}
	\end{abstract}
	\date{\today}

	\maketitle
	
	\section{Introduction} \label{sec1}
	We consider  the following  higher order prescribed curvature problem on  $\mathbb{S}^{N}:$
\begin{equation}\label{1+}
D^{m}\tilde u = \widetilde{K}(y)\tilde u^{m^{*}-1}, \quad \tilde u  > 0, \hbox{ in }  \mathbb{S}^{N},
\end{equation}
where $\widetilde{K}(y)$ is a positive and rotationally symmetric function, $m^{*}=\frac{2N}{N-2m}$ with $m\geq 1$ being a integer and $D^m$ is $2m$ order differential operator given by
$$
D^m=\prod_{i=1}^m \left(-\Delta_g+\frac{1}{4}(N-2i)(N+2i-2)\right),
$$
here $\Delta_g$ is the Laplace-Beltrami operator on $\mathbb{S}^N,$ $\mathbb{S}^N$ is the unit sphere with Riemann metric $g.$

\vskip8pt

In the case of $m=1$, the problem \eqref{1+} is reduced to the following prescribed curvature problem
\begin{equation}\label{2+}
-\Delta_{\mathbb{S}^n} \tilde u+\frac{N(N-2)}{2} \tilde u-\widetilde{K}(y)\tilde u^{\frac{N+2}{N-2}}=0,\quad \tilde u>0, \hbox{ on } \mathbb{S}^N.
\end{equation}
By using the stereo-graphic projection, problem \eqref{2+} is reduced  to the following elliptic problem in $\mathbb{R}^N,$
\begin{equation}\label{2++}
-\Delta u=K(y)u^{\frac{N+2}{N-2}},\quad u>0, \hbox{ in }\mathbb{R}^N, u\in D^{1, 2}(\mathbb{R}^N).
\end{equation}

Because of its geometry background, problem \eqref{2++}  has been extensively studied in the last decades. It is known that (see \cite{Wei2})  \eqref{2++}  does not always admit a solution. Hence we are more interested in the sufficient condition on the curvature function $K(y)$, under which the problem \eqref{2++}  admits a solution. Indeed, there have been  a lot of existence results  obtained in the literature, see for
example, \cite{aap, cgs, Li1, Lin1, Lin2, Yan} and the references therein. In particular, we know that any solution
of \eqref{2++} is radially symmetric if there is an $r_0>0$ such that $K(|y|)$ is nonincreasing in $(0, r_0]$
and nondecreasing in $[r_0,+\infty)$ (see \cite{bianchi}). It is natural to ask wether or not there are non-radial symmetric
solutions to \eqref{2++}. This questions was answered in the  paper of \cite{Wei2}. Recently in \cite{duan-musso-wei}, the authors proved the new type nonradial solution with doubling bubbles. 

\vskip8pt

In general case of any $m\geq1,$ the problem \eqref{2++} is getting more interest  due to its geometry roots and various applications in physics during the last decades. For instance, when
$m = 2$, the problem \eqref{1+} is related to the Paneitz operator, which was introduced by Paneitz
\cite{pan}  for smooth four dimensional Riemannian manifolds and was generalized by Branson \cite{Branson} to
smooth $N$-dimensional Riemannian manifolds. For various  existence results for problems involving  higher order
operator and other related problems, we refer the readers to the papers \cite{ bw, bw2,  Chang1,chang0, efj, ggs, gs2, gs, guo4, guo5, Pucci1, Pucci2}
and the references therein.  It is evident to conclude from these papers that compared with the problems with Laplace operators (that is $m=1$), the problems involving higher order operator
present new and challenging features which make the problem get more complicated.

By using the stereo-graphic projection, problem \eqref{2+} is reduced  to the following elliptic problem in $\mathbb{R}^N,$
\begin{equation}\label{pr1}
(-\Delta)^{m} u=K(y)u^{\frac{N+2m}{N-2m}},\quad u>0, \hbox{ in }\mathbb{R}^N, u\in D^{m, 2}(\mathbb{R}^N).
 \end{equation}

In \cite{guo2}, Guo and Li generalized the results in \cite{Wei2} and showed that the  problem \eqref{pr1} has infinitely many solutions provided  $K $ is radially symmetric  and has a local maximum at some  $r_0 >0$. More precisely, they assumed that there exist  $ r_0 $,  $c_0> 0 $ and $ l \in [2, N-2m)  $ such that
		\begin{equation*}
			K(s)= K(r_0)-c_0|s-r_0|^l+O\big({|s-r_0|^{l+\sigma}}\big),
			\quad s\in(r_0-\delta, r_0+\delta),
		\end{equation*}
		for some  $\sigma, \delta $  small positive constants.  These solutions are obtained by gluing together a large number of  Aubin-Talenti bubbles  (see \cite{Swanson,Wei1})
		\begin{align*}
			U_{x, \Lambda}(y)= c_{N,m} \Big(\frac{\Lambda}{1+\Lambda^2|y-x|^2}\Big)^{\frac{N-2m}{2}}, \, c_{N,m}=\left(\prod_{i=-m}^{m-1} ( N+2i)\right)^{\frac{N-2m}{4m}}.
		\end{align*}
		For any $x \in {\R^N}$ and $\Lambda \in {\R^+}$, these functions solve
		\begin{equation}\label{criticalequation}
			(-\Delta)^{m} u -  u^{\frac{N+2m}{N-2m}}=0,  \quad ~ \text{in}~ {\R}^N.
		\end{equation}
		In fact, up to scaling and transformation,  they are the only positive solutions to \eqref{criticalequation}. Moreover, it known that the kernel of the linear operator associated to \eqref{criticalequation} is spanned by
\begin{equation}\label{Zj}
		Z_i(y) :={\partial U \over \partial y_i}(y), \quad i=1, \cdots , N, \quad
		Z_{N+1}(y) :={N-2m \over 2} U(y)+y\cdot \nabla U(y).
	\end{equation}
Meanwhile, these functions can span the set of the solution to
	\begin{equation}\label{linearized}
		(-\Delta)^m \phi=(m^*-1) U^{m^*-2} \phi,~ \text{in}~ \mathbb{R}^N, \quad \phi \in D^{m,2}(\mathbb{R}^N).
	\end{equation}

Then the main order of the solutions constructed in \cite{guo2} looks like
		\begin{equation*}
			\tilde  u_k \sim \sum_{j=1}^k U_{x_j, \bar \Lambda},
		\end{equation*}
		where  $ \bar \Lambda $ is a positive constant and the points $x_j$ are distributed along the vertices of a regular polygon of $k$ edges in the $(y_1,y_2)$-plane, with $|x_j| \to r_0$ as $k \to \infty$:
		\begin{equation*}
			x_j=\Big(\tilde r \cos\frac{2(j-1)\pi}k, \tilde r\sin\frac{2(j-1)\pi}k, 0,\cdots,0\Big),\quad j=1,\cdots,k,
		\end{equation*}
		with  $ \tilde r \rightarrow  r_0  $ as  $ k \rightarrow \infty $.


Motivated by the work in \cite{duan-musso-wei}, the purpose of this paper is to present a different type of solution to \eqref{pr1} with a more complex concentration structure, which cannot be reduced to a two-dimensional one.

    To present our results, we assume that $K$ satisfies the following conditions:\\[2mm]
    $(\bf K)$ : $K(y)$ is radially symmetric. There are  $ r_0  $ and  $c_0> 0 $ such that
    \begin{equation}
    	K(s)= K(r_0)-c_0|s-r_0|^l+O\big({|s-r_0|^{l+1}}\big),
    	\quad s\in(r_0-\delta, r_0+\delta),
    \end{equation}
    where \begin{equation}\label{assumptionform}
    	l \in ( l_{N,m},N-2m]\cap[2,N-2m],\quad N\ge 2m+3,
    \end{equation}
    and \begin{equation}
        l_{N,m}=\frac{N-2m}{4(N-2m+1)} \Big(-3N+6m-1+(25(N-2m)^2+22(N-2m)+1)^\frac{1}{2}\Big).
    \end{equation}
Let $k$ be an integer number and consider the points below:
	\begin{align*}
		x^{+}_{k,j,r}=r\bigg((1-h^{2})^{1/2}\cos\frac{2(j-1)\pi}{k},(1-h^{2})^{1/2}\sin\frac{2(j-1)\pi}{k},h,{\bf {0}}\bigg), \quad j=1,\cdots,k,
	\end{align*}
	and
	\begin{align*}
		x^{-}_{k,j,r}=r\bigg((1-h^{2})^{1/2}\cos\frac{2(j-1)\pi}{k},(1-h^{2})^{1/2}\sin\frac{2(j-1)\pi}{k},-h,{\bf {0}}\bigg), \quad j=1,\cdots,k,
	\end{align*}
	where  $ {\bf{0}} $ is the zero vector in  $\mathbb{R}^{N-3} $ and  $ h, r  $ are positive parameters.
       For any point $y \in \mathbb{R}^{N}$, we set $y = (y',y'')$,$y'\in\mathbb{R}^{2}$,$y''\in\mathbb{R}^{N-2}$.  Let
		\begin{align}\label{Wrh1}
			W_{r,h,\Lambda} (y)
			&= \sum_{j=1}^kU_{x^{+}_{k,j,r}, \Lambda}  (y) +\sum_{j=1}^k U_{x^{-}_{k,j,r}, \Lambda}  (y), \quad y \in \R^N.
		\end{align}
  and
	\begin{equation}
		\mu_k =k^{\frac{N-2m}{N-2m-l}}.
	\end{equation}
 In this paper, we prove that for any  $k $ large enough problem \eqref{pr1} has  a family of solutions $u_k$ with the main term looks like the following  form:
\begin{equation}\label{approximatesolution2}
	u_k(y) \sim W_{r_k,h_k,\Lambda_k\mu_k}(y).
\end{equation}

    In order to simplify our proof, we normalize the problem first. Without loss of generality, we may assume that $K(r_0) = 1$ and $r_{0}=1$.  Let
 $ v(y)={\mu_k}^{-\frac{N-2m}{2}}u \big(\frac{|y|}{{\mu_k}}\big) $, so that problem \eqref{pr1} becomes
		\begin{align}\label{Prob2}
				(-\Delta)^m v=  K\Big(\frac{|y|}{{\mu_k}}\Big) v^{m^*-1}, \quad v >0,  \quad \text{in}~ \mathbb{R}^N,
				\quad
				v \in D^{m,2}(\mathbb{R}^N).
			\end{align}
   In our proof, we will equivalently construct a family of solutions to \eqref{Prob2}, which are small perturbations of $W_{r,h,\Lambda}$, for any integer $k$ sufficiently large.
	For  $j= 1,\cdots, k $, we divide $\mathbb{R}^N$ into $k$ parts:
	\begin{align*}
		\Omega_j := &\Big\{y=(y_1, y_2, y_3, y'') \in \mathbb{R}^3 \times \mathbb{R}^{N-3}:
		\nonumber\\[2mm]
		& \qquad\Big\langle \frac{(y_1, y_2)}{|(y_1, y_2)|},  \Big(\cos{\frac{2(j-1) \pi}{k}}, \sin{\frac{2(j-1) \pi}{k}}\Big)  \Big\rangle_{\mathbb{R}^2}\geq \cos{\frac \pi k}\Big\}.
	\end{align*}
	where $ \langle , \rangle_{\mathbb{R}^2}$ denote the dot product in $\mathbb{R}^2$. Furthermore, we divide $\Omega_j$ into two parts:
	\begin{align*}
		\Omega_j^+= & \Big\{y:  y=(y_1, y_2, y_3, y'')  \in  \Omega_j, y_3\geq0 \Big\},
		\\[2mm]
		\Omega_j^-= & \Big\{y: y=(y_1, y_2, y_3, y'')  \in  \Omega_j, y_3<0 \Big\},
	\end{align*}
	then
	$$\mathbb{R}^N= \underset{j=1}{\overset{k}{\cup}} \Omega_j,  \quad \Omega_j=  \Omega_j^+\cup  \Omega_j^-$$
	and
	$$ \Omega_j \cap  \Omega_i=\emptyset,  \quad  \Omega_j^+\cap  \Omega_j^-=\emptyset, \qquad \text{if}\quad i\neq j.$$

		\medskip

	Define the symmetric Sobolev space:
	\begin{align*}
		H_{s,k} = \bigg\{ u : &u \hbox{ is  even  in } y_{\ell}, \quad \ell =2,\cdots,N,
		\\&u(r\cos\varphi,r\sin\varphi,y'') = u\bigg(r\cos\bigg(\varphi + \frac{2\pi j}{k}\bigg), r\sin\bigg(\varphi + \frac{2\pi j}{k}\bigg),y''\bigg) ,\hbox{ }j=1,\cdots,k\bigg\},
	\end{align*}
where $\varphi = \arctan{\frac{y_2}{y_1}}$. We define the following norms:
	\begin{align}
		||u||_{*,k} = \sup_{y\in\mathbb{R}^{N}}\bigg(\sum_{j=1}^{k}\bigg(\frac{1}{(1 + |y-x^{+}_{k,j,r}| )^{\frac{N-2m}{2}+\tau}}+\frac{1}{(1 + |y-x^{-}_{k,j,r}| )^{\frac{N-2m}{2}+\tau}}  \bigg)\bigg)^{-1}|u(y)|,
	\end{align}
and	
\begin{align}
	||f||_{**,k} = \sup_{y\in\mathbb{R}^{N}}\bigg(\sum_{j=1}^{k}\bigg(\frac{1}{(1 + |y-x^{+}_{k,j,r}| )^{\frac{N+2m}{2}+\tau}}+\frac{1}{(1 + |y-x^{-}_{k,j,r}| )^{\frac{N+2m}{2}+\tau}}  \bigg)\bigg)^{-1}|f(y)|,
\end{align}
	where $\tau$ is any fixed number such that
	\begin{equation}
		\tau \in \bigg(\frac{N-2m-l}{N-2m},\frac{N-2m-l}{N-2m}+\epsilon_1\bigg),
	\end{equation}  for $\epsilon_1>0$ is a small constant. For the reader's convenience, we will provide a collection of notation. Throughout this paper, we employ  $C,  C_j  $ to denote certain constants and  $ \sigma, \tau,  \zeta_j  $ to denote some small constants or functions.  We also use $ \delta_{ij} $  to denote Kronecker delta function:
	\[ \delta_{ij}= \begin{cases} 1,  \quad \text{if} ~ i= j, \\[2mm]
		0,  \quad \text{if} ~ i \neq j.
	\end{cases} \]
	Furthermore, we  also  employ the notation  by writing  $O(f(r,h)), o(f(r,h))  $ for the functions which satisfy
	\begin{equation*}
		\text{if} \quad g(r,h) \in O(f(r,h)) \quad \text{then}\quad {\lim_{k \to+\infty}} \Bigg|\, \frac{g(r,h)}{f(r,h)} \, \Bigg|\leq C<+\infty,
	\end{equation*}
	and
	\begin{equation*}
		\text{if} \quad g(r,h) \in o(f(r,h)) \quad \text{then}\quad {\lim_{k \to+\infty}} \frac{g(r,h)}{f(r,h)}=0.
	\end{equation*}
	
 The main results of the paper are the following:
	\begin{theorem}\label{main1}
		Suppose that $N \geq 2m+3$, $K(r)$ satisfies $(\bf K)$, $K(r)\geq 0$, $K(r)$ and $K'(r)$ are bounded. Then there is an integer $k_{0}>0$, such that for any integer $k \geq k_{0}$, equation \eqref{pr1} has a solution $u_{k}$ of the form
		\begin{equation}
		    u_{k}(y) = W_{r_{k},h_{k},\Lambda_{k}\mu_k}(y) + \omega_{k}(y),
		\end{equation}
		where $\omega_{k}\in H_{s,k}\bigcap D^{m,2}(\mathbb{R}^{N})\bigcap C(\mathbb{R}^{N})$, and

  \begin{equation*}
      |r_{k}-1| = O\bigg(\frac{1}{\mu_k^{l+\sigma}}\bigg), \ \Lambda_{k}\to \Lambda_{0}>0, \  h_{k}k^{\frac{N-2m-1}{N-2m+1}}\to h_{0}, \ ||\omega_{k}||_{*,k}=O\bigg(\frac{1}{\mu_k^{\frac{l}{2}+\sigma}}\bigg) \ as \ k \rightarrow +\infty,
  \end{equation*}
		for some $\sigma > 0$, where $ \Lambda_0,  h_0 $ are the constants in  \eqref{lambda0}, \eqref{h_0}.
	\end{theorem}

	As a consequence, we have

\begin{theorem}\label{main12}
		Under the same assumptions as in Theorem \ref{main1},  equation \eqref{1+} has infinitely many positive nonradial solutions, which  are invariant under some non-trivial sub-groups of $O(3),$ and their energy can be made arbitrarily large.
	\end{theorem}
	 Throughout of this paper, we assume $(r,h, \Lambda) \in{{\mathscr D}_k} $, where
		\begin{align}\label{definitionofsk}
			{{\mathscr D}_k}
			= \Bigg\{&(r,h,\Lambda) :\,  r\in \Big[k^{\frac{N-2m}{N-2m-l}}-\hat \sigma, k^{\frac{N-2m}{N-2m-l}}+\hat  \sigma \Big], \quad
			\Lambda \in \Big[\Lambda_0-\hat   \sigma,  \Lambda_0+\hat \sigma\Big], \nonumber
			\\[2mm]
			& \qquad \qquad
			h \in \Big[\frac{h_0}{k^{\frac{N-2m-1}{N-2m+1}}}  \Big(1-\hat   \sigma\Big),
			\frac{h_0}{k^{\frac{N-2m-1}{N-2m+1}}} \Big(1+\hat \sigma\Big)\Big] \Bigg\},
		\end{align}
		with  $ \hat \sigma$ is a small fixed number, independent of $k$. For  convenience, we denote
		$$ \quad{\lambda_k}= \frac{h_0}{k^{\frac{N-2m-1}{N-2m+1}}}. $$
	
	The remaining part of this paper will be organized as follows:
In Section \ref{sec3}, we will establish the linearized theory for the linearized  problem and give estimates for the error terms. In Section \ref{sec4}, we shall proof Theorem \ref{main1} by showing there exists a critical point of reduction function  $  F(r,h,\Lambda) .$
		The energy expansions and some tedious computations and useful estimates will be given in Appendices \ref{appendixA}-\ref{appendixB}.


	
	\medskip
	
	\section{Finite dimensional reduction}  \label{sec3}
	
	\medskip
	We consider the following linearized problem:
	\begin{align}\label{lin}
					\begin{cases}
						(-\Delta)^m {\phi}-(m^*-1)K\big(\frac{|y|}{{\mu_k}}\big)
			W_{r,h,\Lambda}^{m^*-2}\phi=f+\sum\limits_{i=1}^k\sum\limits_{\ell=1}^3 \Big({c}_\ell U_{x^{+}_{k,i,r},\Lambda}^{m^*-2}\overline{\mathbb{Z}}_{\ell i}+{c}_\ell U_{x^{-}_{k,i,r},\Lambda}^{m^*-2}\underline{\mathbb{Z}}_{\ell i}\Big)
			\;\;
			\text{in}\;
			\R^N,
						\\[2mm]
						\phi\in \mathbb{E},
					\end{cases}
     \end{align}
	for some constants $c_{\ell} $, where the functions $\overline{\mathbb{Z}}_{\ell j}$ and $\underline{\mathbb{Z}}_{\ell j}$ are given by
		\begin{align*}
			&\overline{\mathbb{Z}}_{1j}=\frac{\partial U_{x^{+}_{k,j,r}, \Lambda}}{\partial r},
			\qquad \qquad
			\overline{\mathbb{Z}}_{2j}=\frac{\partial U_{x^{+}_{k,j,r}, \Lambda}}{\partial h},
			\qquad \qquad
			\overline{\mathbb{Z}}_{3j}=\frac{\partial U_{x^{+}_{k,j,r}, \Lambda}}{\partial \Lambda}, \nonumber
			\\[2mm]
			& \underline{\mathbb{Z}}_{1j}=\frac{\partial U_{x^{-}_{k,j,r}, \Lambda}}{\partial r},
			\qquad \qquad
			\underline{\mathbb{Z}}_{2j}=\frac{\partial U_{x^{-}_{k,j,r}, \Lambda}}{\partial h},
			\qquad \qquad
			\underline{\mathbb{Z}}_{3j}=\frac{\partial U_{x^{-}_{k,j,r}, \Lambda}}{\partial \Lambda},
		\end{align*}
		for $ j= 1, \cdots, k $. Moreover the function $\phi$ belongs to the set $\mathbb{E}$ given by
		\begin{align}\label{SpaceE}
			\mathbb{E}= \Big\{\phi:  \phi\in H_s, \quad \big<U_{x^{+}_{k,j,r}, \Lambda}^{m^*-2}  \overline{\mathbb{Z}}_{\ell j} ,\phi \big> = \big<U_{x^{+}_{k,j,r}, \Lambda}^{m^*-2}  \underline{\mathbb{Z}}_{\ell j} ,\phi \big> = 0,  \quad j= 1, \cdots, k, \quad \ell= 1,2, 3  \Big\}.
            \end{align}
	where $\big<u,v\big> = \displaystyle\int_{\R^N}{u v} .$
	
	\medskip

	\medskip
	\begin{lem}\label{lem2.1}
		Suppose that $\phi_{k}$ solves \eqref{lin} for $f=f_{k}$. If  $\|f_{k}\|_{{**,k}}$ tends to zero as  $k$ tends to infinity, so does  $\|\phi_{k}\|_{{*,k}}$.
	\end{lem}
	\begin{proof}
		We prove the Lemma  by contradiction. Suppose that there exists a sequence of  $(\widetilde{r_k}, h_k, \Lambda_k)\in {{\mathscr D}_k}$, and for $\phi_k$ satisfies \eqref{lin} with  $f=f_k,r= \widetilde{r_k}, h= h_k,  \Lambda= \Lambda_k$, with  $\Vert f _{k}\Vert_{{**,k}}\to 0 $, and  $\|\phi_k\|_{{*,k}}\ge c'>0$. Without loss of generality, we can assume that  $\|\phi_k\|_{{*,k}}=1 $.

		From \eqref{lin}, we know that
		\begin{equation}\label{J123}\begin{split}
				\phi_k(y)
				\,=\,&(m^*-1)\int_{\R^N}\frac1{|z-y|^{N-2m}}K\Big(\frac{|z|}{{\mu_k}}\Big)
				W_{r,h,\Lambda}^{m^*-2}\phi_k(z)\,{\mathrm d}z
				+\int_{\R^N}\frac1{|z-y|^{N-2m}}\,  f_k(z) {\mathrm d}z
				\\[2mm]
				&+\int_{\R^N}\frac1{|z-y|^{N-2m}}\, \sum_{j=1}^k\sum_{\ell=1}^3
				\Big(\, {c_\ell}U_{x^{+}_{k,j,r},\Lambda}^{m^*-2}\overline{\mathbb{Z}}_{\ell j}+{c_\ell}U_{x^{-}_{k,j,r},\Lambda}^{m^*-2}\underline{\mathbb{Z}}_{\ell j}\,\Big) {\mathrm d}z
				\\[2mm]
				:=\,&J_1\,+\,J_2\,+\,J_3.
			\end{split}
		\end{equation}
		Using Lemma \ref{laa3}, we have
		\begin{align*}
			J_1 &\leq C\|{\phi_k}\|_{{*,k}} \, \int_{\R^N}\frac{K\big(\frac{|z|}{{\mu_k}}\big)}{|z-y|^{N-2m}}
			W_{r,h,\Lambda}^{m^*-2} \Big(\sum_{j=1}^k \Big[\frac1{(1+|z-x^{+}_{k,j,r}|)^{\frac{N-2m}{2}+\tau
			}}+\frac1{(1+|z-x^{-}_{k,j,r}|)^{\frac{N-2m}{2}+\tau}}\Big]\Big) \, {\mathrm d}z
			\\[2mm]
			&\leq C\|{\phi_k}\|_{{*,k}} \, \sum_{j=1}^k \Big[\frac1{(1+|z-x^{+}_{k,j,r}|)^{\frac{N-2m}{2}+\tau+\sigma}}+\frac1{(1+|z-x^{-}_{k,j,r}|)^{\frac{N-2m}{2}+\tau+\sigma}}\Big]\nonumber.
		\end{align*}
		
		It follows from Lemma \ref{lemb2} that
		\begin{align*}
			J_2& \le C\|f_k\|_{{**,k}}\int_{\R^N}\frac1{|z-y|^{N-2m}}\sum_{j=1}^k \Big[\frac1{(1+|z-x^{+}_{k,j,r}|)^{\frac{N+2m}2+\tau}}+\frac1{(1+|z-x^{-}_{k,j,r}|)^{\frac{N+2m}2+\tau}}\Big]\,{\mathrm d}z
			\\[2mm]
			& \le C\|f_k\|_{{**,k}} \sum_{j=1}^k \Big[\frac1{(1+|y-x^{+}_{k,j,r}|)^{\frac{N-2m}{2}+\tau}}+\frac1{(1+|y-x^{-}_{k,j,r}|)^{\frac{N-2m}2+\tau}}\Big] \nonumber.
		\end{align*}
		In order to estimate the term  $J_3 $, we will first give the estimates of  $ \overline{\mathbb{Z}}_{\ell j} $ and  $ \underline{\mathbb{Z}}_{\ell j} $
  \begin{equation}\label{estimatekernel}
      |\overline{\mathbb{Z}}_{\ell j}|\le  \frac{C(1+r\delta_{\ell 2})}{(1+|y-x^{+}_{k,j,r}|)^{N-2m+1}},\quad|\underline{\mathbb{Z}}_{\ell j}|\le  \frac{C(1+r\delta_{\ell 2})}{(1+|y-x^{-}_{k,j,r}|)^{N-2m+1}}.
  \end{equation}
	Thus,
		\begin{align}
			\sum_{j=1}^k\, \bigg< \frac1{|z-y|^{N-2m}}\, U_{x^{+}_{k,j,r},\Lambda}^{m^*-2} , \overline{\mathbb{Z}}_{\ell j} \, \bigg>
			& \le C \sum_{j=1}^k\, \int_{\R^N}\frac1{|z-y|^{N-2m}}\, \frac{(1+r \delta_{\ell 2})}{(1+|z-x^{+}_{k,j,r}|)^{N+2m}} \, {\mathrm d}z\nonumber
			\\[2mm]
			& \le C\sum_{j=1}^k\, \frac{(1+r\delta_{\ell 2})}{(1+|y-x^{+}_{k,j,r}|)^{\frac{N-2m}{2}+\tau}}, \quad \text{for} ~\ell= 1, 2, 3,
		\end{align}
		and
		\begin{align}\label{J3}
			\sum_{j=1}^k\, \bigg< \frac1{|z-y|^{N-2m}}\,
			U_{x^{-}_{k,j,r},\Lambda}^{m^*-2} , \underline{\mathbb{Z}}_{\ell j}\bigg>
			\le C  \sum_{j=1}^k\,  \frac{(1+r \, \delta_{\ell 2})}{(1+|y-x^{-}_{k,j,r}|)^{\frac{N-2m}{2}+\tau}}, \quad \text{for} ~\ell= 1, 2, 3.
		\end{align}

		\medskip
		Next, we estimate ${c_\ell}, \ell= 1,2,3 $. Multiply both sides of \eqref{lin} by  $\overline{\mathbb{Z}}_{q 1}, q=1,2,3 $, we obtain that
		\begin{equation}\label{equationofcl}
			\bigg<(-\Delta)^m {\phi_k}-(m^*-1)K\Big(\frac{|y|}{{\mu_k}}\Big)
				W_{r,h,\Lambda}^{m^*-2}{\phi_k} , \overline{\mathbb{Z}}_{q 1}
				\bigg> = \bigg<f_k + \sum_{j=1}^k\sum_{\ell=1}^3 \Big(\, {c_\ell}U_{x^{+}_{k,j,r},\Lambda}^{m^*-2}\overline{\mathbb{Z}}_{\ell j}+{c_\ell}U_{x^{-}_{k,j,r},\Lambda}^{m^*-2}\underline{\mathbb{Z}}_{\ell j}\,\Big), \overline{\mathbb{Z}}_{q 1}\bigg>.
		\end{equation}
		It follows from Lemma \ref{lemb1} that
		\begin{align*}
			\big|\big< f_k , \overline{\mathbb{Z}}_{q 1}\big>\big|
			 \,\le\,& C\|f_k\|_{{**,k}}\,\sum_{j=1}^k  \int_{\R^N}  \frac{1+r \, \delta_{q 2}}{(1+|y-x^{+}_{k,1,r}|)^{N-2m}} \nonumber\\[2mm]&\times\Big[\frac1{(1+|y-x^{+}_{k,j,r}|)^{\frac{N+2m}2+\tau}}
			+\frac1{(1+|y-x^{-}_{k,j,r}|)^{\frac{N+2m}2+\tau}}\Big]\,  \nonumber
			\\[2mm]
			 \,\le &C(1+r \, \delta_{q 2}) \|f_k\|_{{**,k}}.
		\end{align*}
		
		To estimate the left side of \eqref{equationofcl}, integrating by parts and we have
		\begin{align*}
			& \bigg|\bigg<(-\Delta)^m {\phi_k}-(m^*-1)K\Big(\frac{|y|}{{\mu_k}}\Big)
				W_{r,h,\Lambda}^{m^*-2}{\phi_k} , \overline{\mathbb{Z}}_{q 1}
				\bigg> \bigg|\nonumber
			\\[2mm]
			&= \int_{\R^N} \Big[(-\Delta)^m \overline{\mathbb{Z}}_{q 1} -(m^*-1)K\Big(\frac{|y|}{{\mu_k}}\Big)
			W_{r,h,\Lambda}^{m^*-2} \overline{\mathbb{Z}}_{q 1}\Big] {\phi_k}  \nonumber
			\\[2mm]
			&=(m^*-1) \int_{\R^N} \Big[1-K\Big(\frac{|y|}{{\mu_k}}\Big)\Big] W_{r,h,\Lambda}^{m^*-2} \overline{\mathbb{Z}}_{q 1}{\phi_k}+\Big(U_{x^{+}_{k,1,r},\Lambda}^{m^*-2}-W_{r,h,\Lambda}^{m^*-2}\Big)  \overline{\mathbb{Z}}_{q 1}{\phi_k}	\\[2mm]
   &\leq \frac{C}{{\mu_k}^\sigma}(1+r \, \delta_{q 2}) \, \|{\phi_k}\|_{{*,k}}.
		\end{align*}
		On the other hand, there holds
		\begin{align*}
			\sum_{j=1}^k \bigg < \big(\, U_{x^{+}_{k,j,r},\Lambda}^{m^*-2}\overline{\mathbb{Z}}_{\ell j}+U_{x^{-}_{k,j,r},\Lambda}^{m^*-2}\underline{\mathbb{Z}}_{\ell j}\,\Big) ,\overline{\mathbb{Z}}_{q 1} \bigg>
			= \bar{c}_{\ell}\delta_{\ell q} (1+\delta_{q2} r^2)+o(1), \quad {\mbox {as}} \quad k \to \infty.
		\end{align*}
		Note that
		\begin{eqnarray*}
			\big< U_{x^{+}_{k,1,r},\Lambda}^{m^*-2}\overline{\mathbb{Z}}_{\ell 1}  , \overline{\mathbb{Z}}_{q 1} \big> =\left\{\begin{array}{rcl}  0,\qquad \text{if}\quad \ell \neq q,
				&&\\[2mm]
				\space\\[2mm]
				\bar{c}_q (1+\delta_{q2} r^2),\qquad \text{if} \quad\ell=q, \end{array}\right.
		\end{eqnarray*}
		for some constant  $\bar{c}_q >0 $.
		Then we can get
		\begin{align} \label{estimate cl}
			c_{\ell}=  \frac{1+r\delta_{\ell2}}{1+r^2\delta_{\ell2} } O\Bigl(\frac 1{{\mu_k}^\sigma}\|{\phi_k}\|_{{*,k}}+\|f_k\|_{{**,k}} \Bigr)= o(1), \quad {\mbox {as}} \quad k \to \infty.
		\end{align}
		Combining \eqref{J123}-\eqref{J3} and \eqref{estimate cl}, we have
		\begin{equation}
			\label{g8}
			\begin{split}
				|{\phi_k}| \le \Bigl(&\,\|f_k\|_{{**,k}}  \sum_{j=1}^k \Big[\frac1{(1+|y-x^{+}_{k,j,r}|)^{\frac{N-2m}{2}+\tau}}+\frac1{(1+|y-x^{-}_{k,j,r}|)^{\frac{N-2m}{2}+\tau}}\Big]
				\\[2mm]
				&+{\sum_{j=1}^k \Big[\frac1{(1+|y-x^{+}_{k,j,r}|)^{\frac{N-2m}{2}+\tau+\sigma}}+\frac1{(1+|y-x^{-}_{k,j,r}|)^{\frac{N-2m}{2}+\tau+\sigma}}\Big]}\, \Bigr).
			\end{split}
		\end{equation}
		
		Since  $\|{\phi_k}\|_{{*,k}}= 1 $, we obtain from \eqref{g8} that there exist some positive constants  $ \bar{R}, \delta_1 $ such that
		\begin{align} \label{bound1}
			\|\phi_k\|_{L^\infty(B_{\bar{R}}(\overline{x}_l))}  \ge \delta_1> 0,
		\end{align}
		for some  $ l\in \{1,2, \cdots, k\}  $. But $\tilde \phi_k(y)
		= \phi_k(y-x^{+}_{k,j,r})  $  converges uniformly in any compact set to a solution $v$ of
		\begin{equation}\label{reA1}
			(-\Delta)^m v-(m^*-1) U_{0,\Lambda}^{m^*-2} v=0,\quad \text{in}\;\R^N,
		\end{equation}
		for some  $\Lambda\in [L_1,L_2] $  and $v$ is perpendicular to the kernel of \eqref{reA1}. So $v =0$. This is a contradiction to \eqref{bound1}.
		
	\end{proof}

	For the linearized problem \eqref{lin}, we have the following existence, uniqueness results. Furthermore, we can give the estimates of  $ \phi_k $ and  $c_\ell, \ell=1,2,3 $. From Lemma 2.1, using the same argument as in the proof of Proposition 2.2 in \cite{duan-musso-wei} , we can proof the following proposition:
	
	\medskip
	\begin{proposition}\label{p1}
		There exist  $k_0>0  $ and a constant  $C>0 $ such
		that for all  $k\ge k_0 $ and all  $f_k\in L^{\infty}(\R^N) $, problem
		$(\ref{lin}) $ has a unique solution  $\phi_k\equiv {\bf L}_k(f_k) $. Besides,
		\begin{equation}\label{Le}
			\Vert \phi_k\Vert_{{*,k}}\leq C\|f_k\|_{{**,k}},\qquad
			|c_{\ell}|\leq  \frac{C} {1+\delta_{\ell 2} r}\|f_k\|_{{**,k}}, \quad \ell=1,2,3.
		\end{equation}
	\end{proposition}

	\medskip

	Next rewrite problem \eqref{lin} as
	\begin{align}\label{re1}
		\begin{cases}
			(-\Delta)^m \phi_k
			-(m^*-1)K({{\mu_k}^{-1}}|y|)
			W_{r,h,\Lambda}^{m^*-2}\phi_k\\
            \qquad\qquad\qquad\qquad= {{\bf N}}(\phi_k)+{\bf l}_k
			+\sum\limits_{j=1}^k\sum\limits_{\ell=1}^3
			\Big(\, {c_\ell}U_{x^{+}_{k,j,r},\Lambda}^{m^*-2}\overline{\mathbb{Z}}_{\ell j}+{c_\ell}U_{x^{-}_{k,j,r},\Lambda}^{m^*-2}\underline{\mathbb{Z}}_{\ell j}\,\Big)\, \; \text{in}\;
			\R^N,\\[2mm]
			\phi_k \in  \mathbb{E},
		\end{cases}
	\end{align}
	where
	\begin{equation*}
		{{\bf N}}(\phi_k)=K\Big(\frac{|y|}{{\mu_k}}\Big)\Big[\bigl(W_{r,h,\Lambda}+\phi_k \bigr)^{m^*-1}-W_{r,h,\Lambda}^{m^*-1}-(m^*-1)W_{r,h,\Lambda}^{m^*-2}\phi_k\Big],
	\end{equation*}
	and
	\begin{equation*}
		{\bf l}_k=K\Big(\frac{|y|}{{\mu_k}}\Big) W_{r,h,\Lambda}^{m^*-1}-\sum_{j=1}^k \Big(\, U_{x^{+}_{k,j,r},\Lambda}^{m^*-1}+U_{x^{-}_{k,j,r},\Lambda}^{m^*-1} \,\Big).
	\end{equation*}

	In the following, we will use the Contraction Mapping Principle to show that problem \eqref{re1} has a unique solution when $\|\phi_k\|_{{*,k}} $ is small enough. For this purposes, we will first give the estimate  of  ${{\bf N}}(\phi_k) $ and  $ {\bf l}_k $.

	\medskip
	\begin{lemma} \label{Lemma2.3}
		Suppose  $ N\ge 2m+3  $. There exists $C>0$ such that
		\[
		\|{\bf N}(\phi_k)\|_{{**,k}}\le C\|\phi_k\|_{{*,k}}^{\min\{m^*-1, 2\}},
		\]
		for all $\phi_k \in \mathbb{E}$.
	\end{lemma}
	\begin{proof}
		The proof is similar to that of Lemma 2.4 in \cite{guo2}. Here we omit it.
	\end{proof}

	\medskip
	Next, we estimate ${\bf l}_k$.
	\begin{lemma} \label{Lemma2.4}
		Suppose $K(|y|)$ satisfies ${(\bf K)}$ and $N\ge 2m+3$,  $(r,h,\Lambda) \in{{\mathscr D}_k}$.
		There exists $k_0 $ and $C>0$ such that for all $k \geq k_0$
		\begin{align}  \label{estimateforlk}
			\|\, {\bf l}_k \,\|_{{**,k}}  \le  C\max\Bigg\{ \Big(\frac k{\mu_k} \Big)^{
				\frac{N+2m}{2}-\frac{N-2m-l}{N-2m}- \epsilon_1 },  \bigg(\frac1{\mu_k}\bigg)^l \Bigg\},
		\end{align}
		where  $\epsilon_1$ is small constant.
		
	\end{lemma}
	
	\begin{proof}
		We can rewrite ${\bf l}_k$ as
		\[
		\begin{split}
			{\bf l}_k
			\,=\,& K\Big(\frac{|y|}{{\mu_k}}\Big) \Big[\, W_{r,h,\Lambda}^{m^*-1}-\sum_{j=1}^k \Big(\, U_{x^{+}_{k,j,r},\Lambda}^{m^*-1}+U_{x^{-}_{k,j,r},\Lambda}^{m^*-1} \,\Big)  \,\Big]
			\\[2mm]
			&+\sum_{j=1}^k  \Big[\ K\Big(\frac{|y|}{{\mu_k}}\Big)-1  \,\Big]  \Big(\, U_{x^{+}_{k,j,r},\Lambda}^{m^*-1}+U_{x^{-}_{k,j,r},\Lambda}^{m^*-1} \,\Big)
			:=\, M_1+M_2.
		\end{split}
		\]
		
		Assume that  $ y \in \Omega_1^+ $, then we get
		\begin{align*}
			M_1
			&=  K\Big(\frac{|y|}{{\mu_k}}\Big) \Big[\, \Big(\, \sum_{j=1}^k  U_{x^{+}_{k,j,r},\Lambda}+U_{x^{-}_{k,j,r},\Lambda}  \,\Big)^{m^*-1}-\sum_{j=1}^k \Big(\, U_{x^{+}_{k,j,r},\Lambda}^{m^*-1}+U_{x^{-}_{k,j,r},\Lambda}^{m^*-1} \,\Big)  \,\Big]
			\nonumber \\[2mm]
			& \leq  C K\Big(\frac{|y|}{{\mu_k}}\Big)  \Big[U_{x^{+}_{k,1,r},\Lambda}^{m^*-2}\, \Big(\sum_{j=2}^k  U_{x^{+}_{k,j,r},\Lambda}+\sum_{j=1}^k U_{x^{-}_{k,j,r},\Lambda}\Big)+\Big(\sum_{j=2}^k  U_{x^{+}_{k,j,r},\Lambda}+\sum_{j=1}^k U_{x^{-}_{k,j,r},\Lambda}\Big)^{m^*-1}\Big].
		\end{align*}
		Thus, we have
		\begin{align*}
			M_1 \,\le\, & \frac C{(1+|y-x^{+}_{k,1,r}|)^{4m}} \sum_{j=2}^k \frac1{(1+|y-x^{+}_{k,j,r}|)^{N-2m}}
			+\frac C{(1+|y-x^{+}_{k,1,r}|)^{4m}} \sum_{j=1}^k \frac1{(1+|y-x^{-}_{k,j,r}|)^{N-2m}}\\[2mm]
			&+C\Bigl(\sum_{j=2}^k \frac1{(1+|y-x^{+}_{k,j,r}|)^{N-2m}}\Bigr)^{m^*-1}\\[2mm]
			:=&\,M_{11}+M_{12}+M_{13}.
		\end{align*}
		
		For $M_{11}$, when $4m \ge \frac{N+2m}{2} +\tau$, we have
		\begin{align}\label{S11>}
			M_{11}\le\, & C\frac1{(1+|y-x^{+}_{k,1,r}|)^{\frac {N+2m}2+ \tau}} \sum_{j=2}^k \frac1{|x^{+}_{k,j,r}-x^{+}_{k,1,r}|^{N-2m}}  \nonumber
			\\[2mm]
			\,\le\, & C\frac1{(1+|y-x^{+}_{k,1,r}|)^{\frac {N+2m}2+ \tau}} \Big(\frac k{\mu_k} \Big)^{N-2m}	\nonumber
			\\[2mm]
			\,\le\, & C\frac1{(1+|y-x^{+}_{k,1,r}|)^{\frac{N+2m}2+\tau}}\Big(\frac k{\mu_k} \Big)^{
				\frac{N+2m}{2}-\frac{N-2m-l}{N-2m}- \epsilon_1 }.
		\end{align}
		
		When $4m < \frac{N+2m}{2} +\tau$, similar to the proof of Lemma \ref{b.0}, for any  $  1 <  \alpha_1  < N-2m $, we have
		\begin{align*}
			\sum_{j=2}^k \frac1{(1+|y-x^{+}_{k,j,r}|)^{N-2m}}
			\le \frac{C} {(1+|y-x^{+}_{k,1,r}|)^{N-2m-\alpha_1}}\,\sum_{j=2}^k \, \frac1{|x^{+}_{k,j,r}-x^{+}_{k,1,r}|^{\alpha_1}}.
		\end{align*}
		Since  $ \tau  \in ( \frac{N-2m-l}{N-2m},  \frac{N-2m-l}{N-2m}+ \epsilon_1 ) $, we can choose $\alpha_1$ satisfies
		\begin{equation*}
			\frac{N+2m}{2}-\frac{N-2m-l}{N-2m}- \epsilon_1<\alpha_1   =\frac{N+2m}{2}-\tau<N-2m.
		\end{equation*}
		Then
		\begin{align} \label{S11n>5}
			M_{11}			
			& \le  \frac{C} {(1+|y-x^{+}_{k,1,r}|)^{N+2m-\alpha_1}}\,\sum_{j=2}^k \, \frac1{|x^{+}_{k,j,r}-x^{+}_{k,1,r}|^{\alpha_1}} \nonumber
			\\[2mm]
			& \le  \frac{C} {(1+|y-x^{+}_{k,1,r}|)^{N+2m-\alpha_1}} \, \Big(\frac{k}{{\mu_k}\,\sqrt{1-h^2} }\Big)^{\alpha_1}\nonumber
			\\[2mm]
			&\le C\frac1{(1+|y-x^{+}_{k,1,r}|)^{\frac{N+2m}2+\tau}}\Big(\frac k{\mu_k} \Big)^{
				\frac{N+2m}{2}-\frac{N-2m-l}{N-2m}- \epsilon_1 }.
		\end{align}
		Thus, we have
		\begin{equation}\label{S11}
			\|M_{11}\|_{{**,k}}\le C\Big(\frac k{\mu_k} \Big)^{
				\frac{N+2m}{2}-\frac{N-2m-l}{N-2m}- \epsilon_1 }.
		\end{equation}
		Similarly to $M_{11}$, for $M_{12}$, we can obtain that
		\begin{equation}\label{S12}
			\|M_{12}\|_{{**,k}}\le C\Big(\frac k{\mu_k} \Big)^{
				\frac{N+2m}{2}-\frac{N-2m-l}{N-2m}- \epsilon_1 }.
		\end{equation}
		Next, we consider  $M_{13}$. For  $ y \in \Omega_1^+$,
		\begin{align*}
			\sum_{j=2}^k\, \frac1{(1+|y-x^{+}_{k,j,r}|)^{N-2m}}
			& \le  \sum_{j=2}^k \frac1{(1+|y-x^{+}_{k,1,r}|)^{\frac{N-2m}{2}}} \, \frac1{(1+|y-x^{+}_{k,j,r}|)^{\frac{N-2m}{2}}}
			\nonumber\\[2mm]
			& \le  \sum_{j=2}^k\,\frac{C}{|x^{+}_{k,j,r}-x^{+}_{k,1,r}|^{\frac{N-2m}{2}-\frac{N-2m}{N+2m}\tau}}
			\frac 1{(1+|y-x^{+}_{k,1,r}|)^{\frac{N-2m}{2}+\frac{N-2m}{N+2m}\tau}}
			\\[2mm]
			& \le C\Big(\,\frac k {{\mu_k}\sqrt{1-h^2}}\,\Big)^{\frac{N-2m}{2}-\frac{N-2m}{N+2m}\tau} \frac 1{(1+|y-x^{+}_{k,1,r}|)^{\frac{N-2m}{2}+\frac{N-2m}{N+2m}\tau}}.
			\nonumber
		\end{align*}
		Thus we have
		\begin{equation}\label{S13}
			\begin{split}
				M_{13}
				& \le \, \Big(\, \frac k {{\mu_k}\sqrt{1-h^2}}  \,\Big)^{\frac{N+2m}{2}-\tau} \frac{C}{(1+|y-x^{+}_{k,1,r}|)^{\frac{N+2m}{2}+\tau}}
				\\[2mm]
				& \le\frac{C}{(1+|y-x^{+}_{k,1,r}|)^{\frac{N+2m}{2}+\tau}}  \Big(\frac k{\mu_k} \Big)^{
					\frac{N+2m}{2}-\frac{N-2m-l}{N-2m}- \epsilon_1}.
			\end{split}
		\end{equation}
		Combining \eqref{S11}, \eqref{S12}, \eqref{S13}, we obtain
		\begin{equation}\label{S1}
			\begin{split}
				M_{1}
				& \le \, \Big(\, \frac k {{\mu_k}\sqrt{1-h^2}}  \,\Big)^{\frac{N+2m}{2}-\tau} \frac{C}{(1+|y-x^{+}_{k,1,r}|)^{\frac{N+2m}{2}+\tau}}
				\\[2mm]
				& \le\frac{C}{(1+|y-x^{+}_{k,1,r}|)^{\frac{N+2m}{2}+\tau}}  \Big(\frac k{\mu_k} \Big)^{
					\frac{N+2m}{2}-\frac{N-2m-l}{N-2m}- \epsilon_1}.
			\end{split}
		\end{equation}
		
		We now consider the estimate of $M_2$.  For  $ y\in  \Omega_1^+ $, we have
		\begin{align*}
			M_2 \,\le \,& 2\sum_{j=1}^k  \Big[\ K\Big(\frac{|y|}{{\mu_k}}\Big)-1  \,\Big]  \, U_{x^{+}_{k,j,r},\Lambda}^{m^*-1}
			\\[2mm]
			\,=\,& 2\,U_{x^{+}_{k,1,r},\Lambda}^{m^*-1} \,\Big[K \Big(\frac{|y|}{{\mu_k}}\Big)-1\Big]
			+2\,\sum_{j=2}^k \, U_{x^{+}_{k,j,r},\Lambda}^{m^*-1} \, \Big[K\Big(\frac{|y|}{{\mu_k}}\Big)-1\Big]
			\\[2mm]
			:=\,& M_{21}+M_{22}.
		\end{align*}
		If  $|\frac{|y|}{\mu_k}-1|\ge \delta_1,  $ where  $ \delta> \delta_1> 0  $, then
		\begin{align*}
			|y-x^{+}_{k,1,r}| \ge \big||y |-{\mu_k}\big|\,-\,\big|{\mu_k} -|x^{+}_{k,1,r}|\big| \ge \frac{1}{2} \delta_1 {\mu_k}.
		\end{align*}
		As a result, we get
		\begin{align*}
			U_{x^{+}_{k,1,r},\Lambda}^{m^*-1} \,\Big[K\Big(\frac{|y|}{{\mu_k}}\Big)-1\Big]
			& \le \frac{C}{\big(1+|y-x^{+}_{k,1,r}|\big)^{\frac{N+2m} 2+\tau}}  \frac{1}{{\mu_k}^{\frac{N+2m} 2-\tau}}
			\\[2mm]
			& \le \frac{C}{\big(1+|y-x^{+}_{k,1,r}|\big)^{\frac{N+2m} 2+\tau}}  \Big(\frac k{\mu_k} \Big)^{
				\frac{N+2m}{2}-\frac{N-2m-l}{N-2m}- \epsilon_1 }.
		\end{align*}
		Else if $|\frac{|y|}{\mu_k}-1|\le \delta_1,$ then
		\begin{align*}
			\Big[K\Big(\frac{|y|}{{\mu_k}}\Big)-1\Big]  \le& C\Big|\frac{|y|}{\mu_k}-1\Big|^{l} = \frac{C} {{\mu_k}^{l}}||y|-{\mu_k}|^{l}
			\\[2mm]
			\leq&\frac{C} {{\mu_k}^{l}}\Big[\big||y|-|x^{+}_{k,1,r}|\big|^{l}
			\,+\,\big||x^{+}_{k,1,r}|-{\mu_k}\big|^{l}\Big]
			\\[2mm]
			\leq&\frac{C} {{\mu_k}^{l}}\Big[\big||y|-|x^{+}_{k,1,r}|\big|^{l}
			\,+\,\frac{1}{k^{{\bar\theta l}}}\Big].
		\end{align*}
		If $l\ge \frac{N+2m}{2}-\tau$,
		\begin{equation*}
			\frac{ ||y|-|x^{+}_{k,1,r}||^l }{{\mu_k}^{l}}\frac1{(1+|y-x^{+}_{k,1,r}|)^{N+2m}} \le  \frac C{{\mu_k}^{l}}
			\frac1{(1+|y-x^{+}_{k,1,r}|)^{\frac{N+2m}2+\tau}}.
		\end{equation*}
		If $l<\frac{N+2m}{2}-\tau$,
		
		\[
		\begin{split}
			&
			\frac{ ||y|-|x^{+}_{k,1,r}||^l }{{\mu_k}^{l}}\frac1{(1+|y-x^{+}_{k,1,r}|)^{N+2m}}\\
			= & \frac1{{\mu_k}^{\frac {N+2m}2-\tau}}
			\frac1{(1+|y-x^{+}_{k,1,r}|)^{\frac{N+2m}2+\tau}}
			\frac{ ||y|-|x^{+}_{k,1,r}||^l }{{\mu_k}^{ l -\frac {N+2m}2+\tau}}\frac1{(1+|y-x^{+}_{k,1,r}|)^{\frac{N+2m}2-\tau}}\\
			\le & \frac C{{\mu_k}^{\frac{N+2m}2-\tau}}
			\frac1{(1+|y-x^{+}_{k,1,r}|)^{\frac{N+2m}2+\tau}}
			\frac{ ||y|-|x^{+}_{k,1,r}||^{\frac{N+2m}2-\tau} }{(1+|y-x^{+}_{k,1,r}|)^{\frac{N+2m}2-\tau}}\\
			\le & \frac C{{\mu_k}^{\frac{N+2m}2-\tau}}
			\frac1{(1+|y-x^{+}_{k,1,r}|)^{\frac{N+2m}2+\tau}}.
		\end{split}
		\]

		Thus
		
		\begin{equation}\label{13-l2-5-3}
			U_{x^{+}_{k,1,r},\Lambda}^{m^*-1}\Bigg(
			K\bigg(\frac {|y|}{\mu_k}\bigg)-1\Bigg)
			\le  \frac C{{\mu_k}^{\min\{{\frac{N+2m}2-\tau,l}\}}}
			\frac1{(1+|y-x^{+}_{k,1,r}|)^{\frac{N+2m}2+\tau}}.
		\end{equation}
		As a result,
		\begin{equation} \label{S21}
			M_{21} \le  C\Big(\frac k{\mu_k} \Big)^{\min{
					\{\frac{N+2m}{2}-\frac{N-2m-l}{N-2m}- \epsilon_1 },l(N-2m)\}}
			\frac1{(1+|y-x^{+}_{k,1,r}|)^{\frac{N+2m}2+\tau}}.
		\end{equation}
		On the other hand, it is easy to derive that
		\begin{align} \label{S22}
			M_{22}&\le  C\frac{1}{(1+|y-x^{+}_{k,1,r}|)^{\frac{N+2m}2}}\sum_{j=2}^k \frac{1}{(1+|y-x^{+}_{k,j,r}|)^{\frac{N+2m}2}}
			\nonumber \\[2mm]
			& \le C  \frac{1}{(1+|y-x^{+}_{k,1,r}|)^{\frac{N+2m}2+\tau}} \sum_{j=2}^k \frac{1}{|x^{+}_{k,1,r}-x^{+}_{k,j,r}|^{\frac{N+2m}2-\tau}}
			\nonumber \\[2mm]
			& \le  \frac{C}{(1+|y-x^{+}_{k,1,r}|)^{\frac{N+2m}2+\tau}} \Big(\frac k{\mu_k} \Big)^{
				\frac{N+2m}{2}-\frac{N-2m-l}{N-2m}- \epsilon_1 }.
		\end{align}
		Combining \eqref{S21} with \eqref{S22}, we obtain
		\begin{equation*}
			\|M_2\|_{{**,k}} \le C\Big(\frac k{\mu_k} \Big)^{
				\min\{\frac{N+2m}{2}-\frac{N-2m-l}{N-2m}- \epsilon_1,l(N-2m)\} }.
		\end{equation*}
	\end{proof}

	\medskip
	The  solvability theory  for the  linearized  problem \eqref{re1}  can be provided in the following:
	\begin{proposition} \label{pro3}
		Suppose that  $ K(|y|)$ satisfies  ${\bf H}$ and  $N\ge 2m+3$, $(r,h,\Lambda) \in{{\mathscr S}_k}$. There exists an integer  $k_0$ large enough, such that for all  $k \ge k_0$ problem \eqref{re1} has a unique solution $\phi_k$ which satisfies
		\begin{align} \label{estimateforphik}
			\|\phi_k\|_{{*,k}} \le  C\max\Big\{ \frac1{k^{(\frac l{N-2m-l})(
					\frac{N+2m}{2}-\frac{N-2m-l}{N-2m}- \epsilon_1 )}},  \frac1{k^{(\frac {N-2m}{N-2m-l})l}} \Big\},
		\end{align}
		and
		\begin{align}\label{estimateforc}
			|c_{\ell}|\le   \frac C{(1+\delta_{\ell 2}{\mu_k})}  \max\Big\{ \frac1{k^{(\frac l{N-2m-l})(
					\frac{N+2m}{2}-\frac{N-2m-l}{N-2m}- \epsilon_1 )}},  \frac1{k^{(\frac {N-2m}{N-2m-l})l}} \Big\},   \quad{for} ~ \ell=1,2,3.
		\end{align}
	\end{proposition}
	
	\begin{proof}
		Recall that $\mu_k=k^{\frac{N-2m}{N-2m-l}} $, we denote
		\begin{align*}
			\mathcal{A}:= \Bigg\{v:  v\in \mathbb{E} \quad\|v\|_{{*,k}} \le   C \max\Big\{ \frac1{k^{(\frac l{N-2m-l})(
					\frac{N+2m}{2}-\frac{N-2m-l}{N-2m}- \epsilon_1 )}},  \frac1{k^{(\frac {N-2m}{N-2m-l})l}} \Big\} \Bigg\}.
		\end{align*}
		From Proposition \ref{p1}, we know that problem \eqref{re1} is equivalent to the following fixed point problem
		\begin{align*}
			\phi_k= {\bf  L}_k \big({{\bf N}}(\phi_k)+{\bf l}_k\big)= : {\bf A}(\phi_k),
		\end{align*}
		where  ${\bf  L}_k$ is the linear bounded operator defined in Proposition \ref{p1}.

		From Lemma \ref{Lemma2.3} and Lemma \ref{Lemma2.4}, we know that
		\begin{align*}
			\|{\bf A}(\phi_k)\|_{{*,k}}  & \le C  \Big(\|{{\bf N}}(\phi_k)\|_{{**,k}}+\|{\bf l}_k \|_{{**,k}}\Big)
			\nonumber \\[2mm]
			& \le\, O(\| \phi_k\|_{{*,k}}^{1+\sigma} )+  \max\Big\{ \frac1{k^{(\frac l{N-2m-l})(
					\frac{N+2m}{2}-\frac{N-2m-l}{N-2m}- \epsilon_1 )}},  \frac1{k^{(\frac {N-2m}{N-2m-l})l}} \Big\}
			\nonumber \\[2mm]
			&  \le    \max\Big\{ \frac1{k^{(\frac l{N-2m-l})(
					\frac{N+2m}{2}-\frac{N-2m-l}{N-2m}- \epsilon_1 )}},  \frac1{k^{(\frac {N-2m}{N-2m-l})l}} \Big\}.
		\end{align*}
		So the operator ${\bf A}$ maps from  $\mathcal{A}$ to $\mathcal{A}$. Furthermore, we can show that ${\bf A}$ is a contraction mapping.
		In fact, for any  $\phi_{k1}, \phi_{k2} \in \mathcal{A}$, we have
		\begin{align*}
			\|{\bf A}(\phi_{k1})-{\bf A}(\phi_{k2}) \|_{{*,k}}  \le C\|{{\bf N}}(\phi_{k1})-{{\bf N}}(\phi_{k2})\|_{{**,k}}.
		\end{align*}
		Since ${{\bf N}}(\phi_k)$ has a power-like behavior with power greater than one, then we can easily get
		\begin{align*}
			\|{\bf A}(\phi_{k1})-{\bf A}(\phi_{k2}) \|_{{*,k}} \le o(1)\| \phi_{k1}-\phi_{k2}\|_{{*,k}}.
		\end{align*}
        It follows from the contraction mapping principle that there is a unique solution $ \phi_k = {\bf A}(\phi_k)$ in $\mathcal{A}$. The estimates for $ c_{\ell}, \ell=1,2,3  $ comes from \eqref{estimate cl}.
	\end{proof}

	\medskip
	
	\section{Proof of Theorem \ref{main1}}  \label{sec4}

	\begin{proposition} \label{pro2.4}
		Let $ \phi_{r,h, \Lambda}$ be a function obtained in Proposition \ref{pro3} and
		\begin{align*}
			F(r,h,\Lambda):= I( W_{r, h, \Lambda}+\phi_{r,h, \Lambda}),
		\end{align*}
  where
  	\begin{eqnarray}\label{I}
		I(v):=   \left\{\begin{array}{rcl}\displaystyle {\frac{1}{2}} \displaystyle {\int _{\mathbb R^N}|\Delta ^\frac{m}{2} v|^2} -\frac{1}{m^*} \int _{{\mathbb {R}^N}}K\Big(\frac{|y|}{\mu_k}\Big)|v|^{m^*}   ,\qquad \text{if }m \text{ is even},
			&&\\[2mm]
			\space\\[2mm]
		\displaystyle {\frac{1}{2}} \displaystyle {\int _{\mathbb R^N}|\nabla\Delta ^\frac{m-1}{2} v|^2} -\frac{1}{m^*} \int _{\mathbb R^N}K\Big(\frac{|y|}{\mu_k}\Big)|v|^{m^*}   ,\qquad \text{if } m \text{ is odd}. \end{array}\right.		
	\end{eqnarray}
		If $(r,h, \Lambda)$ is a critical point of  $F(r,h,\Lambda)$, then
		\begin{align*}
			v= W_{r,h,\Lambda}+\phi_{r,h,\Lambda}
		\end{align*}
		is a critical point of $I(v)$ in  $H^m(\mathbb{R}^N)$. \qed
	\end{proposition}
	
	\medskip
	We will give the expression of  $F(r,h,\Lambda)$. First, we employ the notation  $\mathcal {C}(r, \Lambda)$ to denote functions which are independent of $h$ and  uniformly bounded.
	
	\begin{proposition}\label{pr0position2,6}
		Suppose that $K(|y|)$ satisfies ${(\bf K)}$ and $N\ge 2m+3 $, $(r,h,\Lambda) \in{{\mathscr D}_k}$.  We have the following expansion as $k \to \infty$
		\begin{align*}
			F(r,h,\Lambda)
			&\,=\, I(W_{r, h, \Lambda})+k O \Big(\frac{1}{k^{\big(\frac{l(N-2m)}{N-2m-l}+\frac{2(N-2m-1)}{N-2m+1}+\sigma\big)}}\Big)
			\nonumber \\[2mm]
			&\,=\, k  A_1 -\frac{k}{\Lambda^{N-2m}} \Big[\,\frac{B_4 k^{N-2m}}{(r \sqrt{1-h^2})^{N-2m}}\,+\, \frac{B_5 k}{r^{N-2m} h^{N-2m-1} \sqrt {1-h^2}}\,\Big]
			\nonumber \\[2mm]
			& \quad+ k \Big[\frac{A_2}{\Lambda^{l}  k^{\frac{(N-2m)l}{N-2m-l}}}
			+\frac{A_3}{\Lambda^{l-2}  k^{\frac{(N-2m)l}{N-2m-l}}}({\mu_k}-r)^2\Big]
			+k  \frac{\mathcal {C}(r, \Lambda)}{k^{\frac{l(N-2m)}{N-2m-l}}}({\mu_k} -r)^{3}
			\nonumber\\[2mm]
			& \quad
			+k  \frac{\mathcal {C}(r, \Lambda)}{k^{\frac{l(N-2m)}{N-2m-l}+\sigma}}
			+k O\Big(\frac{1}{k^{\big(\frac{l(N-2m)}{N-2m-l}+\frac{2(N-2m-1)}{N-2m+1}+\sigma\big)}}\Big),
		\end{align*}
		where  $A_1, A_2, A_3,  B_4, B_5$ are  positive constants.
	\end{proposition}
	\begin{proof}
		Similar to the proof of Proposition  $3.1 $  in \cite{guo2}, we omit that of Proposition \ref{pr0position2,6} here.
		\end{proof}

		\medskip
		Next, we will give the expansions of $\frac{\partial F(r,h, \Lambda)}{\partial \Lambda}$ and $ \frac{\partial F(r,h, \Lambda)}{\partial h} $.
		\begin{proposition}\label{lambdadaoshu}
			Suppose that $K(|y|)$ satisfies ${(\bf K)}$ and $N\ge 2m+3 $, $(r,h,\Lambda) \in{{\mathscr D}_k}$.
			We have the following expansion for $k \to \infty$
			\begin{align}\label{thu25jun}
				\frac{\partial F(r,h, \Lambda)}{\partial \Lambda}
				&= \frac{k (N-2m)}{\Lambda^{N-2m+1}} \Big[\frac{B_4 k^{N-2m}}{(r \sqrt{1-h^2})^{N-2m}}\,+\, \frac{B_5 k}{r^{N-2m} h^{N-2m-1} \sqrt {1-h^2}} \Big]
				\nonumber\\[2mm]
				&\quad-k \Big[\, \frac{l A_2}{\Lambda^{l+1} k^{\frac{(N-2m)l}{N-2m-l}}}
				+\frac{(l-2)A_3}{\Lambda^{l-1}  k^{\frac{(N-2m)l}{N-2m-l}}} ({\mu_k}-r)^2\,\Big]
				+kO \, \Big(\frac1{k^{\frac{(N-2m)l}{N-2m-l}+ \sigma}}\Big),
			\end{align}
   and
   \begin{align} \label{frach}
				\frac{\partial F(r,h, \Lambda)}{\partial h }
				\,=\,   &-\frac{k}{\Lambda^{N-2m}}\Big[\, (N-2m) \frac{B_4 k^{N-2m}}{ r^{N-2m} (\sqrt{1-h^2})^{N-2m+2}} h-(N-2m-1) \frac{B_5 k}{r^{N-2m} h^{N-2m} \sqrt {1-h^2}} \,\Big]\nonumber\\[2mm]
    &+kO\Big(\frac{1}{k^{\big(\frac{l(N-2m)}{N-2m-l}+\frac{(N-2m-1)}{N-2m+1}+M_{m,N,l}-\epsilon_0\big)}}\Big),
			\end{align}
			where  $A_2, A_3, B_4, B_5$ are  positive constants,
   \begin{align}\label{sigma}
					M_{m,N,l}:=&\min \bigg\{\frac{(N-2m)(l-1)}{N-2m-l}-\frac{N-2m-1}{N-2m+1},\frac{l}{N-2m-l}\bigg(4m-\frac{2(N-2m-l)}{N-2m}-\frac{N-2m}{l}\bigg),\nonumber\\[2mm]
						   &\quad\quad\quad\frac{N-2m-1}{N-2m+1},\frac{3l+2m-N}{N-2m-l}-\frac{N-2m-1}{N-2m+1} \bigg\}.
			\end{align}
   and $\epsilon_{0}$ is a small constant.

		\end{proposition}
		\begin{proof}
			The proof of \eqref{thu25jun} is similar to the proof of Proposition 4.4 in \cite{guo2}. For brevity, We omit it here and focus on the proof of \eqref{frach}.
			Notice that $F(r,h, \Lambda)\,=\, I( W_{r, h, \Lambda}+\phi_{r,h, \Lambda})$ , there holds
			\begin{align}
				&\frac{\partial F(r,h, \Lambda)}{\partial h }
				\nonumber\\[2mm]
				&    = \left\langle  I'( W_{r, h, \Lambda}+\phi_{r,h, \Lambda}), \frac{\partial  W_{r, h, \Lambda}  } {\partial  h} \right\rangle
				+  \left\langle  I'( W_{r, h, \Lambda}+\phi_{r,h, \Lambda}), \frac{\partial  \phi_{r,h, \Lambda}  } {\partial  h}\right\rangle
				\nonumber\\[2mm]
				&    =   \left\langle  I'( W_{r, h, \Lambda}+\phi_{r,h, \Lambda}), \frac{\partial  W_{r, h, \Lambda}  } {\partial  h} \right\rangle  +  \left\langle \sum\limits_{j=1}^k\sum\limits_{\ell=1}^3
				\Big(\, {c_\ell}U_{x^{+}_{k,j,r},\Lambda}^{m^*-2}\overline{\mathbb{Z}}_{\ell j}+{c_\ell}U_{x^{-}_{k,j,r},\Lambda}^{m^*-2}\underline{\mathbb{Z}}_{\ell j}\,\Big),    \frac{\partial  \phi_{r,h, \Lambda}  } {\partial  h}\right\rangle.
			\end{align}
			
			Noting that $ \displaystyle{\int_{{\mathbb{R}}^N} U_{x^{+}_{k,j,r}, \Lambda}^{m^*-2}  \overline{\mathbb{Z}}_{\ell j}  \phi_{r,h, \Lambda}}=  \displaystyle{\int_{{\mathbb{R}}^N} U_{x^{-}_{k,j,r}, \Lambda}^{m^*-2}  \underline{\mathbb{Z}}_{\ell j}  \phi_{r,h, \Lambda}=0}$, we have
			\begin{align*}
				\left\langle U_{x^{+}_{k,j,r},\Lambda}^{m^*-2}\overline{\mathbb{Z}}_{\ell j}, \frac{\partial  \phi_{r,h, \Lambda}  } {\partial  h} \right\rangle
				\,=\, - \left\langle \frac { \partial (U_{x^{+}_{k,j,r},\Lambda}^{m^*-2}\overline{\mathbb{Z}}_{\ell j})}{\partial h},\phi_{r,h, \Lambda} \right\rangle,
			\end{align*}
			\begin{align*}
				\left\langle U_{x^{-}_{k,j,r},\Lambda}^{m^*-2}\overline{\mathbb{Z}}_{\ell j},    \frac{\partial  \phi_{r,h, \Lambda}  } {\partial  h}\ \right\rangle
				\, =\,  -\left\langle \frac { \partial (U_{x^{+}_{k,j,r},\Lambda}^{m^*-2}\overline{\mathbb{Z}}_{\ell j})}{\partial h},      \phi_{r,h, \Lambda} \right\rangle.
			\end{align*}
			Thus,
			\begin{align}
				&\Bigg\langle  \sum\limits_{j=1}^k
				\Big(\, {c_\ell}U_{x^{+}_{k,j,r},\Lambda}^{m^*-2}\overline{\mathbb{Z}}_{\ell j}+{c_\ell}U_{x^{-}_{k,j,r},\Lambda}^{m^*-2}\underline{\mathbb{Z}}_{\ell j}\,\Big),    \frac{\partial  \phi_{r,h, \Lambda}  } {\partial  h}\Bigg\rangle
				\nonumber\\[2mm]
				&\leq C |c_\ell| \|  \phi_{r,h, \Lambda} \|_{{*,k}}   \sum\limits_{i=1}^k \int_{\R^N}     \frac { \partial (U_{x^{+}_{k,i,r},\Lambda}^{m^*-2}\overline{\mathbb{Z}}_{\ell i})}{\partial h}  \Biggl(
				\sum_{j=1}^k \Big[\frac1{(1+|y-x^{+}_{k,j,r}|)^{\frac{N-2m}{2}+\tau
				}}+\frac1{(1+|y-x^{-}_{k,j,r}|)^{\frac{N-2m}{2}+\tau
				}}\Big] \Biggr)
				\nonumber\\[2mm]
				&\leq
				C |c_\ell| \|  \phi_{r,h, \Lambda} \|_{{*,k}}
				\nonumber\\[2mm]
				& \quad \times
				\sum\limits_{i=1}^k \int_{\R^N}    \frac {r(1+\delta_{\ell 2}{\mu_k}) } {(1+ |y-x^{+}_{k,i,r}| )^{N+2m+1}  }   \Biggl(
				\sum_{j=1}^k \Big[\frac1{(1+|y-x^{+}_{k,j,r}|)^{\frac{N-2m}{2}+\tau
				}}+\frac1{(1+|y-x^{-}_{k,j,r}|)^{\frac{N-2m}{2}+\tau
				}}\Big] \Biggr)
				\nonumber\\[2mm]
				&\leq  C  {\mu_k}   \frac1{k^{(\frac l{N-2m-l})(
						N+2m-2\frac{N-2m-l}{N-2m}- 2\epsilon_1 )}},
			\end{align}
			where we combined \eqref{estimateforphik},\eqref{estimateforc} and  the inequalities
			\begin{align*}
				\Bigg|  \frac{ \partial  \big( U_{x^{+}_{k,i,r},\Lambda}^{m^*-2}\overline{\mathbb{Z}}_{ \ell i} \big)}{\partial h}  \Bigg|
				\le C   \frac {{\mu_k}(1+\delta_{\ell 2} r) } {(1+ |y-x^{+}_{k,i,r}| )^{N+2m+1}  }\quad \text{for}~i=1, \cdots,k,  \ell=1,2,3.
			\end{align*}
			
			On the other hand,
			\begin{align}\label{I'parW}
				& \Bigg\langle  I'( W_{r, h, \Lambda}+\phi_{r,h, \Lambda}), \frac{\partial  W_{r, h, \Lambda}  } {\partial  h} \Bigg\rangle
				\nonumber\\[2mm]
				& =   \frac{\partial I( W_{r, h, \Lambda}) } {\partial h}  +   (m^*-1) \int_{\mathbb{R}^N} K\Big(\frac{|y|}{{\mu_k}}\Big)  W_{r, h, \Lambda}^{m^*-2} \frac{\partial  W_{r, h, \Lambda}  } {\partial  h}  \phi_{r,h, \Lambda}   + O\Big( \int_{\mathbb{R}^N} \phi_{r,h, \Lambda}^2  \Big).
			\end{align}
   Moreover, from the the decay property of $K(|y|)$ and  orthogonality of $\phi_{r, h, \Lambda}$, we can show that the second term in \eqref{I'parW} is small.
			
			\begin{align*}
				& \int_{\mathbb{R}^N} K\Big(\frac{|y|}{{\mu_k}}\Big)  W_{r, h, \Lambda}^{m^*-2} \frac{\partial  W_{r, h, \Lambda}  } {\partial  h}  \phi_{r,h, \Lambda}
				\nonumber\\[2mm]
				& =    \int_{\mathbb{R}^N} K\Big(\frac{|y|}{{\mu_k}}\Big)  \Big[  W_{r, h, \Lambda}^{m^*-2} \frac{\partial  W_{r, h, \Lambda}  } {\partial  h}   -  \sum_{i=1}^k \big(U_{x^{+}_{k,i,r},\Lambda}^{m^*-2} \overline{\mathbb{Z}}_{2i}  + U_{x^{-}_{k,i,r},\Lambda}^{m^*-2}  \underline{\mathbb{Z}}_{2i} \big) \Big] \phi_{r,h, \Lambda}
				\nonumber\\[2mm]
				& \quad + \sum_{i=1}^k    \int_{\mathbb{R}^N} \Big[K\Big(\frac{|y|}{{\mu_k}}\Big)-1\Big]  \big(U_{x^{+}_{k,i,r},\Lambda}^{m^*-2} \overline{\mathbb{Z}}_{2i}  + U_{x^{-}_{k,i,r},\Lambda}^{m^*-2} \underline{\mathbb{Z}}_{2i} \big)  \phi_{r,h, \Lambda}
				\nonumber\\[2mm]
				& = 2k   \int_{\Omega^+_1}   K\Big(\frac{|y|}{{\mu_k}}\Big)  \Big[  W_{r, h, \Lambda}^{m^*-2} \frac{\partial  W_{r, h, \Lambda}  } {\partial  h}   -  \sum_{i=1}^k \big(U_{x^{+}_{k,i,r},\Lambda}^{m^*-2} \overline{\mathbb{Z}}_{2i}  + U_{x^{-}_{k,i,r},\Lambda}^{m^*-2} \underline{\mathbb{Z}}_{2i} \big)  \Big] \phi_{r,h, \Lambda}
				\nonumber\\[2mm]
				& \quad + 2k \int_{\mathbb{R}^N} \Big[K\Big(\frac{|y|}{{\mu_k}}\Big)-1\Big]   U_{x^{+}_{k,1,r},\Lambda}^{m^*-2}\overline{\mathbb{Z}}_{21}     \phi_{r,h, \Lambda}.
			\end{align*}
			Using the expression of $W_{r, h, \Lambda}$, we obtain that
			\begin{align*}
				 \int_{\Omega^+_1}&K\Big(\frac{|y|}{{\mu_k}}\Big)\,\Big[  W_{r, h, \Lambda}^{m^*-2} \frac{\partial  W_{r, h, \Lambda}  } {\partial  h}
				\,-\,\sum_{i=1}^k \big(U_{x^{+}_{k,i,r},\Lambda}^{m^*-2} \overline{\mathbb{Z}}_{2i}  + U_{x^{-}_{k,i,r},\Lambda}^{m^*-2}  \underline{\mathbb{Z}}_{2i} \big)\Big] \phi_{r,h, \Lambda}
				\nonumber\\[2mm]
				\le &C \int_{\Omega^+_1}  \Big[ U_{x^{+}_{k,1,r},\Lambda}^{m^*-2}    \big( \sum_{j=2 }^k  \overline{\mathbb{Z}}_{2j}
				\,+\,\sum_{j=1 }^k   \underline{\mathbb{Z}}_{2j}    \big)  + \big( \sum_{i=2}^k U_{x^{+}_{k,i,r},\Lambda}^{m^*-2} \overline{\mathbb{Z}}_{2i}
				+\sum_{i=1}^k  U_{x^{-}_{k,i,r},\Lambda}^{m^*-2}\underline{\mathbb{Z}}_{2i} \big) \Big] \phi_{r,h, \Lambda}
				\nonumber\\[2mm]
				\le& C\Big(\frac{k}{{\mu_k}}\Big)^{\frac {N+2m} 2 -\tau} \int_{\Omega^+_1}   \frac{\mu_k}{\big(1+|y-x^{+}_{k,1,r}|\big)^{\frac{N+2m}{2}+1+\tau}} \phi_{r,h, \Lambda}
				\nonumber\\[2mm]
				  \le& C \Big(\frac{k}{{\mu_k}}\Big)^{\frac {N+2m} 2 -\tau}   \| \phi_{r,h, \Lambda}   \|_{{*,k}}  \int_{\Omega^+_1}   \frac{\mu_k}{\big(1+|y-x^{+}_{k,1,r}|\big)^{\frac{N+2m}{2}+1+\tau}}
   \\&\times\Bigl(
				\sum_{j=1}^k \Big[\frac1{(1+|y-x^{+}_{k,j,r}|)^{\frac{N-2m}{2}+\tau
				}}+\frac1{(1+|y-x^{-}_{k,j,r}|)^{\frac{N-2m}{2}+\tau
				}}\Big] \Bigr)
				\nonumber\\[2mm]
				  \le& C {\mu_k} \Big(\frac{k}{{\mu_k}}\Big)^{\frac {N+2m} 2 -\tau} \| \phi_{r,h, \Lambda}   \|_{{*,k}} \nonumber\\[2mm]
				\leq&  C  {\mu_k}    \frac1{k^{(\frac l{N-2m-l})(
						N+2m-2\frac{N-2m-l}{N-2m}- 2\epsilon_1 )}}.
			\end{align*}
			And it's easy to show that
			\begin{equation*}
				\int_{\mathbb{R}^N} \Big[K\Big(\frac{|y|}{{\mu_k}}\Big)-1\Big]   U_{x^{+}_{k,1,r},\Lambda}^{m^*-2} \,\overline{\mathbb{Z}}_{21}\,\phi_{r,h, \Lambda}\leq  C  {\mu_k}   \frac1{k^{(\frac l{N-2m-l})(
						N+2m-2\frac{N-2m-l}{N-2m}- 2\epsilon_1 )}}.
			\end{equation*}
			Combining all the results above, we can get
			\begin{equation}\label{ee}
				\frac{\partial F(r,h, \Lambda)}{\partial h } = \frac{\partial I(W_{r,h,\Lambda})}{\partial h}+  kO\Big( {\mu_k}    \frac1{k^{(\frac l{N-2m-l})(N+2m-2\frac{N-2m-l}{N-2m}- 2\epsilon_1 )}}\Big).
			\end{equation}
			which, together with Proposition \ref{func2} and Lemma \ref{B22}, leads to \eqref{frach}.
		\end{proof}

		\medskip
  Using Proposition \ref{frach}, we can identify the order of parameters $(r,h,\Lambda)$. Let $\Lambda_0$ be the solution of
  \[\frac{B_4(N-2m)}{\Lambda^{N-2m+1}}- \frac{A_2 l}{\Lambda^{l+1}}=0,\]
  which gives that
		\begin{align} \label{lambda0}
			\Lambda_0=\Big[\,\frac{(N-2m) B_4}{A_2 l}\,\Big]^{\frac 1 {N-2m-l}}.
		\end{align}
		And let
             \begin{align} \label{h_0}
			{\lambda_k} = \frac{h_0}{k^{\frac{N-2m-1}{N-2m+1}}}, \quad \text{with}   ~h_0=  \Big[\frac{(N-2m-1) B_5}{(N-2m) B_4}\Big]^{\frac 1 {N-2m+1}},
		\end{align}
             which solves  $$\Big[(N-2m)  B_4 k^{N-2m} h-(N-2m-1) \frac{B_5 k}{h^{N-2m}}\Big]=0.$$
		Correspondingly, we define
		\begin{align}\label{Dk}
			{D_k}
			=\Bigg\{&(r,h,\Lambda) :\,  r\in \Big[k^{\frac{N-2m}{N-2m-l}}-\frac{1}{{k^{\bar \theta}}}, k^{\frac{N-2m}{N-2m-l}}+\frac{1}{{k^{\bar \theta}}} \Big], \quad
			\Lambda \in \Big[\Lambda_0-\frac{1}{{k^{\frac{3  \bar \theta}{2}}}},  \Lambda_0+\frac{1}{{k^{\frac{3  \bar \theta}{2}}}}\Big],
			\nonumber\\[2mm]
			& \qquad \qquad h \in \Big[\frac{h_0}{k^{\frac{N-2m-1}{N-2m+1}}} \Big(1-\frac{1}{k^{\bar \theta}}\Big),\frac{h_0}{k^{\frac{N-2m-1}{N-2m+1}}} \Big(1+\frac{1}{k^{\bar \theta}}\Big)\Big] \Bigg\},
		\end{align}
		for  $\bar \theta \le M_{m,N,l}$.  In fact, $ {D_k}$ is a subset of ${{\mathscr D}_k}$.   We will find a critical point of $F(r, h, \Lambda)$ in ${D_k}$.
		
		Let
		\begin{align}\label{barF}
			\bar  F(r,h,\Lambda)=-F(r,h,\Lambda),
		\end{align}
		and
		\[t_2= k (-A_1+\eta_1),
		\quad t_1
		= k  \Big(-A_1 -\big( \frac{A_2}{\Lambda_0^{l}} - \frac{B_4}{\Lambda_0^{N-2m}}  \big) \frac{1}{k^{\frac{(N-2m)l}{N-2m-l}}} -\frac{1}{k^{\frac{(N-2m) l}{N-2m-l}+\frac{5 \bar \theta} 2}}\Big),  \]
		where $\eta_1>0 $ small. We also define the energy level set
		\[  \bar F^{t}= \Big\{(r,h,\Lambda) \big|~(r,h,\Lambda) \in {D_k}, ~ \bar F (r,h,\Lambda) \le{t} \Big\}.  \]
		
		Consider the following  gradient flow system
		\begin{equation*}
			\begin{cases}
				\frac{{\mathrm d} r}{{\mathrm d} t}=-{\bar F}_r , &t>0;\\[2mm]
				\frac{{\mathrm d} h}{{\mathrm d} t}=-{\bar F}_h , &t>0;\\[2mm]
				\frac{{\mathrm d} \Lambda}{{\mathrm d} t}=-{\bar F}_\Lambda, &t>0;\\[2mm]
				(r,h,\Lambda) \big|_{t=0}  \in \bar F^{t_2}.
			\end{cases}
		\end{equation*}
		Then we have the next proposition:
		
		\begin{proposition} \label{5.73.3}
			The flow would not leave  $ {D_k} $ before it reaches  $ \bar F^{t_1}. $
		\end{proposition}
		\begin{proof}
			There are   three  cases that the flow tends to leave  $ {D_k} $:
			\\[2mm]
			\item[ {\bf Case  1.} ]
			$|r-{\mu_k}|= \frac{1}{k^{\bar \theta}}  $ and  $|1-{\lambda_k}^{-1} h| \le \frac{1}{k^{\bar \theta}}, \quad|\Lambda-\Lambda_0|\le \frac{1}{k^{\frac{3  \bar \theta}{2}}}  $;
			\\[2mm]
			~\item[ {\bf Case  2.} ] $|1-{\lambda_k}^{-1} h|=  \frac{1}{k^{\bar \theta}}  $  when  $|r-{\mu_k}| \le \frac{1}{k^{\bar \theta}}, \quad|\Lambda-\Lambda_0|\le \frac{1}{k^{\frac{3  \bar \theta}{2}}} $;
			\\[2mm]
			~ \item[ {\bf Case  3.} ]  $| \Lambda-\Lambda_0|=  \frac{1}{k^{\frac{3  \bar \theta}{2}}} $ when  $|r-{\mu_k}| \le \frac{1}{k^{\bar \theta}}, \quad|1-{\lambda_k}^{-1} h| \le \frac{1}{k^{\bar \theta}}. $
			\\[2mm]

			First consider {\bf Case 1} . Since  $|\Lambda-\Lambda_0|\le \frac{1}{k^{\frac{3  \bar \theta}{2}}}  $, it is easy to derive that
			\begin{align} \label{decomlambda} \begin{split}
					\Big(\frac{B_4}{\Lambda^{N-2m}} - \frac{A_2}{\Lambda^{l}}\Big)
					&= \Big(\frac{B_4}{\Lambda_0^{N-2m}} - \frac{A_2}{\Lambda_0^{l}}\Big)+O(|\Lambda-\Lambda_0|^2)\\[2mm]
					&=\Big(\frac{B_4}{\Lambda_0^{N-2m}} - \frac{A_2}{\Lambda_0^{l}}\Big)+O\Big(\frac{1}{k^{3 \bar \theta}}\Big).
				\end{split}
			\end{align}
			Combining  \eqref{barF}, \eqref{decomlambda},\eqref{FF},  we can obtain that
			\begin{align}\label{F}
				\bar F(r,h,\Lambda)  &=-k  A_1+k \Big[\frac{B_4}{\Lambda_0^{N-2m} k^{\frac{(N-2m)l}{N-2m-l}}} -
				\frac{A_2}{\Lambda_0^{l}  k^{\frac{(N-2m)l}{N-2m-l}}} \Big] \nonumber
				\\[2mm]
				& \quad
				- k  \frac{A_3}{\Lambda_0^{l-2}  k^{\frac{(N-2m) l}{N-2m-l}+2 \bar\theta}}
				+O\Big(\frac{1}{k^{\frac{(N-2m) l}{N-2m-l}+2 \bar \theta+2\epsilon} }\Big) <t_1.
			\end{align}
			
			In {\bf Case 2}, we claim that it's impossible for the flow  $\big(r(t), h(t), \Lambda(t)\big)$ leaves  ${D_k}$. Because if $1-{\lambda_k}^{-1} h=  \frac{1}{k^{\bar \theta}} $, then from  \eqref{barF},\eqref{frach1} and $\bar{\theta}=\bar{\sigma}-\epsilon$, we have
			\begin{align}
				\frac{\partial \bar F(r,h,\Lambda)}{\partial h} = -\frac{k}{\Lambda^{N-2m}}\Big[\,  \frac{2B_7}{\Lambda^{N-2m}  k^{\frac{(N-2m)l}{N-2m-l}+\frac{(N-2m-1)}{N-2m+1} +\bar\theta}} \,\Big]  + O\Big(\frac{1}{k^{\frac{(N-2m) l}{N-2m-l}+ \frac{(N-2m-1)}{N-2m+1} +2 \bar \theta}}\,\Big) <0.
			\end{align}
			On the other hand, if  $1-{\lambda_k}^{-1} h=  -\frac{1}{k^{\bar \theta}} $
			\begin{align}
				\frac{\partial \bar F(r,h,\Lambda)}{\partial h} = \frac{k}{\Lambda^{N-2m}}\Big[\,  \frac{2B_7}{\Lambda^{N-2m}  k^{\frac{(N-2m)l}{N-2m-l}+\frac{(N-2m-1)}{N-2m+1} +\bar\theta}} \,\Big]  + O\Big(\frac{1}{k^{\frac{(N-2m) l}{N-2m-l}+ \frac{N-2m-1} {N-2m+1} +2 \bar \theta}}\,\Big) >0.
			\end{align}		
			So it's  impossible for  the flow leaves  ${D_k}$ in {\bf Case 2}.\\
				\medskip
				
    Finally, we consider {\bf Case 3}.
				If  $\Lambda=\Lambda_0+\frac{1}{k^{\frac{3  \bar \theta}{2}}} $, from
				\eqref{thu25jun} and \eqref{barF},  there exists a constant  $ C_1 $ such that
				\[
				\frac{\partial \bar F(r,h,\Lambda)}{\partial \Lambda}
				= k\Big[\,C_1 \frac{1}{k^{\frac{(N-2m) l}{N-2m-l}+\frac 32 \bar \theta}}
				+O\Big(\frac{1}{k^{\frac{(N-2m) l}{N-2m-l}+2 \bar \theta}}\Big)\,\Big]
				>0.
				\]
				On the other hand, if  $\Lambda=\Lambda_0-\frac{1}{k^{\frac{3  \bar \theta}{2}}}$, there exists a constant  $C_2$ such that
				\[
				\frac{\partial \bar F(r,h,\Lambda)}{\partial \Lambda}
				= k\Big[\,-C_2 \frac{1}{k^{\frac{(N-2m) l}{N-2m-l}+\frac 32 \bar \theta}}
				+O\Big(\frac{1}{k^{\frac{(N-2m) l}{N-2m-l}+2 \bar \theta}}\,\Big)\Big]<0.
				\]
				Hence the flow  $ \big(r(t), h(t), \Lambda(t)\big)$ does not leave  ${D_k}$ when  $| \Lambda-\Lambda_0|=  \frac{1}{k^{\frac{3  \bar \theta}{2}}}$.
				
				Combining above results, we conclude that the flow would not leave  ${D_k}$ before it reach  $ \bar{F}^{t_1}$.
			\end{proof}

			\vspace{3mm}
			Proposition \ref{5.73.3} plays an important role in proving Theorem \ref{main1}, now let us give the proof of it. \\[2mm]
			{\bf{Proof of Theorem \ref{main1}}:}
			According to  Proposition \ref{pro2.4}, in order to show Theorem \ref{main1},  we only need to  show that function  $\bar F(r,h, \Lambda)$, and thus $F(r,h, \Lambda)$,  has a critical point in  ${D_k}$.
			
			Define
			\[
			\begin{split}
				G=\Bigl\{\gamma :\quad &\gamma(r,h,\Lambda)= \big(\gamma _1(r,h,\Lambda),\gamma _2(r,h,\Lambda),\gamma _3(r,h,\Lambda)\big)\in{D_k},(r,h,\Lambda)\in{D_k}; \\[2mm]
				&\gamma(r,h,\Lambda)=(r,h,\Lambda), \;\text{if} ~|r-{\mu_k}|=  \frac{1}{k^{\bar \theta}}  \Bigr\}.
			\end{split}
			\]
			
			Let
			\[
			{\bf c}=\inf_{\gamma  \in G}\max_{(r,h,\Lambda)\in{D_k}} \bar F \big(\gamma(r,h,\Lambda)\big).
			\]
			
			We claim that  $ {\bf c}  $ is a critical value of  $ \bar F(r,h, \Lambda)  $ and can be achieved by some  $(r,h, \Lambda) \in{D_k} $. By the minimax theory, we need to show that
			\medskip
			\begin{itemize}
				
				\item[(i)]  $ t_1< {\bf c} < t_2 $;
				\medskip

				\item[(ii)] \label{itemii}  $\underset{{|r-{\mu_k}|=  \frac{1}{k^{\bar \theta}}}}{\sup}
				{\bar F}\big(\gamma(r,h,\Lambda)\big)<t_1,\;\forall\; \gamma \in G. $
				
			\end{itemize}
			
			To prove (ii),  let $\gamma\in G$. Then for any  $\bar r$ with
			$|\bar r-{\mu_k}|=  \frac{1}{k^{\bar \theta}}$, we have $\gamma(\bar r,h,\Lambda)= (\bar r,h,\bar\Lambda)$ for some $\bar\Lambda$.
			Thus, by \eqref{F},
			
			\[
			\bar F(\gamma(\bar r,h,\Lambda))= \bar F(\bar  r,h,\bar \Lambda)<t_1.
			\]
			
			Now we prove (i).  It is easy to see that
			
			\[
			{\bf c}<t_2.
			\]
			
			For any $\gamma=(\gamma_1,\gamma_2,\gamma_3)\in G$. Then $\gamma_1(r,h,\Lambda)= r$, if
			$|r-{\mu_k}|=  \frac{1}{k^{\bar \theta}}$. Define
			
			\[
			\tilde \gamma_1(r)=\gamma_1(r,h_0,\Lambda_0).
			\]
			Then  $\tilde \gamma_1(r)= r$, if $|r-{\mu_k}|=  \frac{1}{k^{\bar \theta}}$.
			So, there is a $\bar r\in ({\mu_k}-\frac1{k^{\bar\theta}}, {\mu_k}+\frac1{k^{\bar\theta}})$, such that
			
			\[
			\tilde \gamma_1(\bar r)={\mu_k}.
			\]
			Let $\bar \Lambda= \gamma_2(\bar r,h_0,\Lambda_0),\bar h=\gamma_3(\bar r, h_0, \Lambda_0)$. Then from \eqref{FF}
			
			\begin{align*}
				&\max_{(r,h,\Lambda)\in {D_k}}\bar F(\gamma(r,h,\Lambda))\ge \bar F(\gamma(\bar r,h_0,\Lambda_0))=
				\bar F({\mu_k}, \bar h,\bar \Lambda)\\
				=&- k  A_1 +k \Big[\frac{B_4}{\bar \Lambda^{N-2m} k^{\frac{(N-2m)l}{N-2m-l}}}+\frac{B_6}{\bar \Lambda^{N-2m} k^{\frac{(N-2m)l}{N-2m-l}+\frac{2(N-2m-1)}{N-2m+1}}}
				+\frac{B_7}{\bar \Lambda^{N-2m}  k^{\frac{(N-2m)l}{N-2m-l}+\frac{2(N-2m-1)}{N-2m+1}}} (1-{\lambda_k}^{-1} \bar h)^2\Big]\\
				&-	k \frac{A_2}{\bar \Lambda^{l}  k^{\frac{(N-2m)l}{N-2m-l}}}
				+k O\big (\frac{1}{k^{\frac{(N-2m)l}{N-2m-l}+\frac{2(N-2m-1)}{N-2m+1} + \min\{3 \bar \theta,\sigma\}}}\Big )\\
				> &-k  A_1 +k \frac{B_4}{ \Lambda_{0}^{N-2m} k^{\frac{(N-2m)l}{N-2m-l}}}
				-	k \frac{A_2}{ \Lambda_{0}^{l}  k^{\frac{(N-2m)l}{N-2m-l}}}
				+k O\big (\frac{1}{k^{\frac{(N-2m)l}{N-2m-l}+3 \bar \theta}}\Big )
				>t_1.
			\end{align*}

                 By scaling, we equivalently know that problem \eqref{pr1} has a solution of the form mentioned in Theorem \ref{main1}
                \begin{align*}
				u_k(y)={\mu_k}^{\frac{N-2m} 2} \Big[W_{r_k, h_k,\Lambda_k}({\mu_k} y)+\phi_k({\mu_k} y)\Big]=W_{r_k, h_k,\Lambda_k\mu_k}(y)+\omega_k.
			\end{align*}
   \qed\\
   \medskip

			\appendix

			\section{Energy expansions}\label{appendixA}
			This section is devoted to the computation of the expansion for the energy functional  $I(W_{r,h,\Lambda})  $.  We first give the following Lemma.
			\begin{lemma}  \label{express7}
				$N\ge 2m+3 $ and  $(r,h,\Lambda) \in{{\mathscr D}_k}$.
				For $k \to \infty,i=1,\cdots,k$, we have the following expansions:
				\begin{align} \label{express3}
					\sum_{j=2}^k  \frac{1}{|x^{+}_{k,1,r}-x^{+}_{k,j,r}|^{N-2m}} =  \frac{k^{N-2m}}{\big(r \sqrt{1-h^2}\big)^{N-2m}}  \big(B_1+\zeta_1(k)\big),
				\end{align}
				\begin{equation}
					\begin{split} \label{express4}
						& \sum_{j=1}^k  \frac{1}{|x^{+}_{k,1,r}-x^{-}_{k,j,r}|^{N-2m}} = \frac{B_2 k}{r^{N-2m} h^{N-2m-1} \sqrt {1-h^2}} \, \big(1+\zeta_2(k)\big)+\frac{\zeta_1(k) k^{N-2m}}{\big(r \sqrt{1-h^2}\big)^{N-2m}},
					\end{split}
				\end{equation}
    	\begin{align}\label{a33}
					\int_{\R^N}  U_{x^{+}_{k,1,r}, \Lambda}^{m^*-1}  U_{x^{+}_{k,i,r}, \Lambda}
					= \frac{B_0}{\Lambda^{N-2m}|x^{+}_{k,1,r}-x^{+}_{k,i,r}|^{N-2m}}+O\Big(\frac 1 {|x^{+}_{k,1,r}-x^{+}_{k,i,r}|^{N-2m+2-\epsilon_0}}\Big),
				\end{align}
				and
				\begin{align}\label{a34}
					\int_{\R^N}  U_{x^{+}_{k,1,r}, \Lambda}^{m^*-1}  U_{x^{-}_{k,i,r}, \Lambda}
					=\frac{B_0}{\Lambda^{N-2m}|x^{+}_{k,1,r}-x^{-}_{k,i,r}|^{N-2m}}+O\Big(\frac 1 {|x^{+}_{k,1,r}-x^{-}_{k,i,r}|^{N-2m+2-\epsilon_0}}\Big),  \quad 
				\end{align}
				where
    \begin{equation}
         B_0= \displaystyle\int_{\mathbb{R}^N}  \frac{1}{(1+z^{2})^{\frac{N+2m} 2}},
    \end{equation}
				\begin{align} \label{B1sigmak}
					 B_1=  \frac 2 {(2\pi)^{N-2m}} \sum_{j=1}^{\infty} \frac 1 {j^{N-2m}},
					\quad B_2= \frac 1 {2^{N-2m-1} \pi}\int_0^{+\infty}\,\frac{1}{\big(s^2+1\big)^{\frac{N-2m}{2}}}\,{\mathrm d}s,
				\end{align}
				\begin{align} \label{sigmak}
					\zeta_1(k)=
					\begin{cases}
						O\big(\frac{1}{k^{2}}\big),  \quad  N\ge 2m+4,
						\\[2mm]
						O\big(\frac{\ln k}{k^2}\big), \quad  N = 2m+3,
					\end{cases}  \quad \zeta_2(k)= O\big((hk)^{-1}\big).
				\end{align}
    and $\epsilon_0$ is constant small enough.
			\end{lemma}

  \begin{proof}
  Without loss of generality, we assume $ k $ is even. Noting that for  $ \frac{1}{2} < \kappa_1 \leq \kappa_2\leq 1 $,  we have
  $ \kappa_1 \frac{j \pi}{k}
 \leq  \sin {\frac{j\pi}{k}} \leq \kappa_2 \frac{j \pi}{k}, \quad  ~j \in  \big\{1, \cdots,  \frac k2 \big\}.$
Direct computations show that
\begin{align*}
& \sum_{j=2}^k  \frac{1}{|x^{+}_{k,1,r}-x^{+}_{k,j,r}|^{N-2m}}
= \sum_{j=1}^k \Big(\frac{1}{2r \sqrt{1-h^2} \sin{\frac{j\pi}{k}}}\Big)^{N-2m} \nonumber
\\[2mm]
&=  \sum_{j=1}^{\frac k2} \,\Big(\frac{1}{2r \sqrt{1-h^2} \sin{\frac{j \pi}{k}}}\Big)^{N-2m}+\sum_{j= \frac k2+1}^{k} \,\Big(\frac{1}{2r \sqrt{1-h^2} \sin{\frac{j \pi}{k}}}\Big)^{N-2m}
\\[2mm]
&=\sum_{j=1}^{{[\frac k 6]}} \Big(\frac{1}{2r \sqrt{1-h^2} \sin{\frac{j\pi}{k}}}\Big)^{N-2m}
+O \Big(\frac{k}{\big(2r \sqrt{1-h^2}\big)^{N-2m}}\Big)+\sum_{j= \frac k2+1}^{k} \,\Big(\frac{1}{2r \sqrt{1-h^2} \sin{\frac{j \pi}{k}}}\Big)^{N-2m}\\[2mm]
&= \Big(\frac{k}{r \sqrt{1-h^2}}\Big)^{N-2m} \Big(\frac{1}{2}B_{1}+\zeta_1(k)\Big)+\sum_{j= \frac k2+1}^{k} \,\Big(\frac{1}{2r \sqrt{1-h^2} \sin{\frac{j \pi}{k}}}\Big)^{N-2m}\\[2mm]
&=\Big(\frac{k}{r \sqrt{1-h^2}}\Big)^{N-2m} \Big(B_{1}+\zeta_1(k)\Big).
\end{align*}
where  $  B_{1}$ and  $\zeta_1(k)  $ are defined in \eqref{sigmak}, thus we proved \eqref{express3}.

Similarly, we can obtain
\begin{equation*}
\begin{split}
& \sum_{j=1}^k  \frac{1}{|x^{+}_{k,1,r}-x^{-}_{k,j,r}|^{N-2m}}
= \sum_{j=1}^k \frac{1}{\Big(2r \big[(1-h^2) \sin^2 {\frac{(j-1) \pi}{k}}+h^2\big]^{\frac{1}{2}}\Big)^{N-2m}}
 \\[2mm]
 &= \frac{2}{(2r h)^{N-2m}} \sum_{j=1}^{\frac k2} \frac{1}{\Big(\frac{(1-h^2)}{h^2}\frac{(j-1)^2 \pi^2}{k^2}+1\Big)^{\frac{N-2m}{2}}}+\zeta_1(k)  O \Big(\Big(\frac{k}{r \sqrt{1-h^2}}\Big)^{N-2m}\Big)\\[2mm]
 &=\frac{2hk}{(2r h)^{N-2m}{\sqrt{1-h^2} \pi}}\, \int_0^{+\infty}\,\frac{1}{\big(s^2+1\big)^{\frac{N-2m}{2}}}\,{\mathrm d}s  \Big(1+O\big( (kh)^{-1} \big)\Big)+\zeta_1(k)  O \Big(\Big(\frac{k}{r \sqrt{1-h^2}}\Big)^{N-2m}\Big)\\[2mm]
 &=\frac{h}{(r h)^{N-2m}}  \frac{B_2 k}{\sqrt{1-h^2}} \Big(1+\zeta_2 (k) \Big)
+\zeta_1(k)O \Big(\frac{ k^{N-2m}}{\big(r \sqrt{1-h^2}\big)^{N-2m}}\Big).
\end{split}
\end{equation*}
where  $ B_2 $ and  $ \zeta_2(k) $  are defined in \eqref{B1sigmak}, \eqref{sigmak}.

 Next, we prove \eqref{a33}. Define $ \overline{d}_j =|x^{+}_{k,1,r}-x^{+}_{k,j,r}|, ~ \underline{d}_j =|x^{+}_{k,1,r}-x^{-}_{k,j,r}| $ for  $j=1, \cdots, k $. We consider
\begin{align} \label{A10}
 \int_{\mathbb{R}^N}  U_{x^{+}_{k,1,r}, \Lambda}^{m^*-1} U_{x^{+}_{k,i,r}, \Lambda}
 = &\int_{\mathbb{R}^N}  \frac{\Lambda^{\frac{N+2m} 2}}{(1+\Lambda^2|y-x^{+}_{k,1,r}|^{2})^{\frac{N+2m} 2}}  \frac{\Lambda^{\frac{N-2m}{2}}}{(1+\Lambda^2|y-x^{+}_{k,i,r}|^{2})^{\frac{N-2m}{2}}}  \nonumber
 \\[2mm]
  =& \int_{B_{\frac{\overline{d}_ i} 4}(x^{+}_{k,1,r})}\frac{\Lambda^{\frac{N+2m} 2}}{(1+\Lambda^2|y-x^{+}_{k,1,r}|^{2})^{\frac{N+2m} 2}}  \frac{\Lambda^{\frac{N-2m}{2}}}{(1+\Lambda^2|y-x^{+}_{k,i,r}|^{2})^{\frac{N-2m}{2}}}\\[2mm]\nonumber
  &+\int_{\mathbb{R}^N \setminus {B_{\frac{\overline  d_i} 4}(x^{+}_{k,1,r})}} \frac{\Lambda^{\frac{N+2m} 2}}{(1+\Lambda^2|y-x^{+}_{k,1,r}|^{2})^{\frac{N+2m} 2}}  \frac{\Lambda^{\frac{N-2m}{2}}}{(1+\Lambda^2|y-x^{+}_{k,i,r}|^{2})^{\frac{N-2m}{2}}}\\[2mm]\nonumber
  :=&N_1+N_2.
\end{align}
For the first term $N_1$, we have
\begin{align} \label{cal1}
 N_1 = &\int_{B_{\frac{\Lambda \overline{d}_i} 4}(0)}  \frac{1}{(1+z^{2})^{\frac{N+2m} 2}}  \frac{1}{(1+z^{2}+2\Lambda \langle z , x^{+}_{k,1,r}-x^{+}_{k,i,r} \rangle +\Lambda^2|x^{+}_{k,1,r}-x^{+}_{k,i,r}|^2)^{\frac{N-2m}{2}}}
\nonumber \\[2mm]\nonumber
 =& \frac 1 {\Lambda^{N-2m}|x^{+}_{k,1,r}-x^{+}_{k,i,r}|^{N-2m}} \int_{B_{\frac{\Lambda \overline{d}_i} 4}(0)}  \frac{1}{(1+z^{2})^{\frac{N+2m} 2}} \Bigg(1-\frac{N-2m}{2}  \frac{1+z^{2}+2 \Lambda\langle z , x^{+}_{k,1,r}-x^{+}_{k,i,r} \rangle}{\Lambda^{2}|x^{+}_{k,1,r}-x^{+}_{k,i,r}|^{2}}
 \\[2mm]
 &+O \Big(\Big(\frac{1+z^{2}+2  \Lambda\langle z , x^{+}_{k,1,r}-x^{+}_{k,i,r} \rangle}{\Lambda^{2}|x^{+}_{k,1,r}-x^{+}_{k,i,r}|^{2}}
\Big)^2\Big)\Bigg)\\[2mm]\nonumber
  =&\frac{B_0}{\Lambda^{N-2m}|x^{+}_{k,1,r}-x^{+}_{k,i,r}|^{N-2m}}
 +O\Big(\frac 1 {|x^{+}_{k,1,r}-x^{+}_{k,i,r}|^{N}}\Big)+O\Big(\frac 1 {|x^{+}_{k,1,r}-x^{+}_{k,i,r}|^{N-2m+2-\epsilon_0}}\Big)\\[2mm]\nonumber
 =&\frac{B_0}{\Lambda^{N-2m}|x^{+}_{k,1,r}-x^{+}_{k,i,r}|^{N-2m}}+O\Big(\frac 1 {|x^{+}_{k,1,r}-x^{+}_{k,i,r}|^{N-2m+2-\epsilon_0}}\Big).
\end{align}
 where  $B_0= \displaystyle \int_{\mathbb{R}^N}  \frac{1}{(1+z^{2})^{\frac{N+2m} 2}}  $.
For the second term $N_2$, we can easily get
\begin{align}  \label{A.16}
N_2 = O\Big(\frac 1 {|x^{+}_{k,1,r}-x^{+}_{k,i,r}|^{N-\epsilon_0}}\Big).
\end{align}
Combining \eqref{A10},\eqref{cal1},\eqref{A.16}, then we can get \eqref{a33}. Similarly, we can prove \eqref{a34}
 \end{proof}

			\begin{lemma}  \label{express9}
				Suppose that  $ K(|y|) $ satisfies   ${(\bf K)} $ and  $N\ge 2m+3 $, $(r,h,\Lambda) \in{{\mathscr D}_k}$.
				We have the expansion for $k \to \infty$
				\begin{align*}
					I(W_{r,h,\Lambda})
					\,=\,& kA_1-k \int_{\R^N}U_{x^{+}_{k,1,r}, \Lambda}^{m^*-1}\Big(\sum_{j=2}^k  U_{x^{+}_{k,j,r}, \Lambda}+\sum_{j=1}^k U_{x^{-}_{k,j,r}, \Lambda}\Big)
					\\[2mm]
					&+k\Big[\frac{A_2}{\Lambda^{l}{\mu_k}^{l}}
					+\frac{A_3}{\Lambda^{l-2}{\mu_k}^{l}}({\mu_k}-r)^2\Big]
					+k\frac{\mathcal {C}(r, \Lambda)}{{\mu_k}^{l}}  \big({\mu_k}-r\big)^{3}
					\\[2mm]
					&+k\frac{\mathcal {C}(r, \Lambda)}{{\mu_k}^{l+\sigma}}+k O\Big(\Big(\frac{k}{{\mu_k}}\Big)^{N-\epsilon_0}\Big)
					+k O\Big(\frac{1}{k^{l}} \Big(\frac{k}{{\mu_k}}\Big)^{N-2m}\Big),
				\end{align*}
				where $\mathcal {C}(r, \Lambda)$ denotes function independent of $h$ and should be order of $O(1)$,
				\begin{equation}\label{A1A2}
					A_1=\Big(1-\frac{2}{m^*}\Big) \int_{\R^N}|U_{0,1}|^{m^*},
					\quad
					A_2=\frac{2 c_0}{m^*} \int_{\R^N}|y_1|^{l}U_{0,1}^{m^*},
				\end{equation}
				\begin{equation}\label{B_0B_1}
					A_3=\frac{c_0 l(l-1)}{m^*}\int_{\R^N}|y_1|^{l-2}U_{0,1}^{m^*},
				\end{equation}
				and  $ \epsilon_0  $ is constant can be chosen small enough.
			\end{lemma}

			\begin{proof}
				Recalling  the definition of  $I(v) $ as in \eqref{I}, then we obtain that
				\begin{align}
					I(W_{r,h,\Lambda})
					\,=\,& \frac{1}{2} \int_{\mathbb{R}^N}|\nabla W_{r,h,\Lambda}|^2- \frac{1}{m^*} \int_{\mathbb{R}^N} K\Big(\frac{|y|}{{\mu_k}}\Big) W_{r,h,\Lambda}^{m^*}
					\nonumber\\[2mm]
					:=\,&I_1-I_2.
				\end{align}
				
				According to the expression of $W_{r,h,\Lambda}$, we have
				\begin{align}\label{estimateofI1}
					I_1
					\,=\,&\frac{1}{2}\sum_{j=1}^k \sum_{i=1}^k \int_{\mathbb{R}^N} (-\Delta)^m \Big(U_{x^{+}_{k,j,r}, \Lambda}+U_{x^{-}_{k,j,r}, \Lambda}\Big) \Big(U_{x^{+}_{k,i,r}, \Lambda}+U_{x^{-}_{k,i,r}, \Lambda}\Big)
					\nonumber\\[2mm]
					\,=\,&k \sum_{j=1}^k \int_{\mathbb{R}^N} \Big(U_{x^{+}_{k,1,r},\Lambda}^{m^*-1}U_{x^{+}_{k,j,r}, \Lambda}+U_{x^{-}_{k,1,r}, \Lambda}^{m^*-1} U_{x^{+}_{k,j,r}, \Lambda}\Big)
					\nonumber\\[2mm]
					\,=\,&k\int_{\mathbb{R}^N} \Big(U_{0,1}^{m^*}+\sum_{j=2}^kU_{x^{+}_{k,1,r}, \Lambda}^{m^*-1} U_{x^{+}_{k,j,r},\Lambda}\Big)
					\,+\,k\int_{\mathbb{R}^N}\sum_{j=1}^k U_{x^{-}_{k,1,r},\Lambda}^{m^*-1} U_{x^{+}_{k,j,r}, \Lambda}
					\nonumber\\[2mm]
					\,=\,& k\int_{\mathbb{R}^N}U_{0,1}^{m^*} +k\int_{\mathbb{R}^N} U_{x^{+}_{k,1,r}, \Lambda}^{m^*-1}\Big(\sum_{j=2}^kU_{x^{+}_{k,j,r}, \Lambda}\,+\,\sum_{i=1}^k  U_{x^{-}_{k,j,r}, \Lambda}\Big).
				\end{align}
				
				For  $I_2$,  using the symmetry of function  $W_{r,h,\Lambda}$,  we have
				\begin{align}\label{estimateofI2}
					I_2
					\,= \, &\frac{2k}{m^*} \int_{\Omega_1^{+}} K\Big(\frac{|y|}{{\mu_k}}\Big)  W_{r,h,\Lambda}^{m^*}
					\nonumber\\[2mm]
					\,= \, &  \frac{2k}{m^*} \int_{\Omega_1^{+}} K\Big(\frac{|y|}{{\mu_k}}\Big)  \Bigg\{U_{x^{+}_{k,1,r}, \Lambda}^{m^*}+m^* U_{x^{+}_{k,1,r}, \Lambda}^{m^*-1}  \Big(\sum_{j=2}^k U_{x^{+}_{k,j,r}, \Lambda}
					+\sum_{j=1}^k U_{x^{-}_{k,j,r}, \Lambda}\Big)
					\nonumber
					\Bigg\}+kO\Big(\Big(\frac{k}{{\mu_k}}\Big)^{N-\epsilon_0}\Big)
					\nonumber\\
					\,:=\,&\frac{2k}{m^*} \Big(I_{21}+I_{22}\Big)+kO\Big(\Big(\frac{k}{{\mu_k}}\Big)^{N-\epsilon_0}\Big).
				\end{align}

				
				
				For  $I_{21} $, we can rewrite it as following
				\begin{align*}
					I_{21}
					&=  \int_{\Omega_1^{+}} U_{x^{+}_{k,1,r}, \Lambda}^{m^*}+\int_{\Omega_1^{+}}  \Big[K\Big(\frac{|y|}{{\mu_k}}\Big)-1\Big]  U_{x^{+}_{k,1,r}, \Lambda}^{m^*}
					\nonumber\\[2mm]
					&= \int_{\mathbb{R}^N} U_{0,1}^{m^*}+\int_{\Omega_1^{+}}  \Big[K\Big(\frac{|y|}{{\mu_k}}\Big)-1\Big]  U_{x^{+}_{k,1,r}, \Lambda}^{m^*}+O \Big(\Big(\frac{k}{{\mu_k}}\Big)^{N}\Big).
				\end{align*}
				Furthermore, we  obtain
				\begin{align*}
					\int_{\Omega_1^{+}}  \Big[K\Big(\frac{|y|}{{\mu_k}}\Big)-1\Big]  U_{x^{+}_{k,1,r}, \Lambda}^{m^*}
					=& \int_{\Omega_1^{+} \cap \big\{y :\,|\frac{|y|}{\mu_k}-1|\ge \delta  \big\}}\Big[K\Big(\frac{|y|}{{\mu_k}}\Big)-1\Big]  U_{x^{+}_{k,1,r}, \Lambda}^{m^*}
					\\[2mm]
					&+\int_{\Omega_1^{+} \cap \big\{y :\,|\frac{|y|}{\mu_k}-1|\le \delta  \big\}} \Big[K\Big(\frac{|y|}{{\mu_k}}\Big)-1\Big]  U_{x^{+}_{k,1,r}, \Lambda}^{m^*}.
				\end{align*}
				If   $|\frac{|y|}{\mu_k}-1|\ge \delta  $, there holds
				\begin{align*}
					|y-x^{+}_{k,1,r}| \ge \big||y |-{\mu_k}\big|\,-\,\big|{\mu_k} -|x^{+}_{k,1,r}|\big|\ge \frac{1}{2} \delta {\mu_k}.
				\end{align*}
				Thus we can easily get
				\begin{align*}
					\int_{\Omega_1^{+} \cap \big\{y :\,|\frac{|y|}{\mu_k}-1|\ge \delta  \big\}} \Big[K\Big(\frac{|y|}{{\mu_k}}\Big)-1\Big]  U_{x^{+}_{k,1,r}, \Lambda}^{m^*} \le  \frac{C}{{\mu_k}^{N-\epsilon_0}}.
				\end{align*}
				If  $|\frac{|y|}{\mu_k}-1|\le \delta  $,  recalling the decay property of  $ K $,  we can obtain that
				\begin{align*}
					&\int_{\Omega_1^{+} \cap \big\{y :\,|\frac{|y|}{\mu_k}-1|\le \delta  \big\}} \Big[K\Big(\frac{|y|}{{\mu_k}}\Big)-1\Big]  U_{x^{+}_{k,1,r}, \Lambda}^{m^*}
					\\[2mm]
					&= -c_0 \frac{1}{{\mu_k}^{l}} \int_{\Omega_1^{+} \cap \big\{y :\,|\frac{|y|}{\mu_k}-1|\le \delta  \big\}}\big||y|-{\mu_k} \big|^{l} \, U_{x^{+}_{k,1,r}, \Lambda}^{m^*}
					\\[2mm]
					& \quad+O\Big(\frac{1}{{\mu_k}^{l+\sigma} }  \int_{\Omega_1^{+} \cap \big\{y :\,|\frac{|y|}{\mu_k}-1|\le \delta  \big\}}\big||y|-{\mu_k} \big|^{l+\sigma} \, U_{x^{+}_{k,1,r}, \Lambda}^{m^*} \Big)
					\\[2mm]
					&= -c_0 \frac{1}{{\mu_k}^{l}} \int_{\R^N}\big||y|-{\mu_k} \big|^{l} \, U_{x^{+}_{k,1,r}, \Lambda}^{m^*}
					+O \Big(\int_{\R^N \setminus B_{\frac{{\mu_k}}{k}}(x^{+}_{k,1,r})}  \Big(\frac{|y|^{l}}{{\mu_k}^{l}}+1\Big)  U_{x^{+}_{k,1,r},\Lambda}^{m^*}\Big)
					\\[2mm]
					& \quad+O\Big(\frac{1}{{\mu_k}^{l+\sigma} } \int_{\Omega_1^{+} \cap \big\{y :\,|\frac{|y|}{\mu_k}-1|\le \delta  \big\}}\big||y|-{\mu_k} \big|^{l+\sigma} \, U_{x^{+}_{k,1,r}, \Lambda}^{m^*} \Big)
					\\[2mm] &
					= -c_0 \frac{1}{{\mu_k}^{l}} \int_{\R^N}
					\big||y+x^{+}_{k,1,r}|-{\mu_k} \big|^{l} \, U_{0, \Lambda}^{m^*}
					\\[2mm]
					& \quad+O\Big(\frac{1}{{\mu_k}^{l+\sigma} } \int_{\Omega_1^{+} \cap \big\{y :\,|\frac{|y|}{\mu_k}-1|\le \delta  \big\}}\big||y|-{\mu_k} \big|^{l+\sigma} \, U_{x^{+}_{k,1,r}, \Lambda}^{m^*} \Big)+O  \Big(\Big(\frac{k}{{\mu_k}}\Big)^{N-\epsilon_0}\Big).
				\end{align*}
				Furthermore, recalling  $|x^{+}_{k,1,r}|= r $ and using the symmetry property,  we have
				\begin{align*}
					\int_{\R^N} \big||y+x^{+}_{k,1,r}|-{\mu_k} \big|^{l}U_{0,\Lambda}^{m^*}
					=  \int_{\R^N} \big||y+e_1 r|-{\mu_k} \big|^{l}U_{0,\Lambda}^{m^*},
				\end{align*}
				where  $ e_1=(1, 0,\cdots, 0) $.
				
				We  get
				\begin{align*}
					&\int_{\R^N}||y+x^{+}_{k,1,r}|-{\mu_k}|^{l}U_{0,\Lambda}^{m^*}\\[2mm]
					=&\int_{\R^N}|y_1|^{l}U_{0,\Lambda}^{m^*}+\frac12 l(l-1) \int_{\R^N}|y_1|^{l-2}U_{0,\Lambda}^{m^*}({\mu_k} -r)^2
					+\mathcal {C}(r, \Lambda)({\mu_k} -r)^{3},
				\end{align*}
				here $\mathcal {C}(r, \Lambda)$ denote functions which are independent of  $h$ and can be absorbed in  $O(1) $.

				Similarly, we can also have the following expression
				\begin{align*}
					&O\Big(\frac{1}{{\mu_k}^{l+\sigma} } \int_{\Omega_1^{+} \cap \big\{y :\,|\frac{|y|}{\mu_k}-1|\le \delta  \big\}}\big||y|-{\mu_k} \big|^{l+\sigma} \, U_{x^{+}_{k,1,r}, \Lambda}^{m^*} \Big)
					\\[2mm]
					&=O\Big(\frac{1}{{\mu_k}^{l+\sigma} } \int_{\R^N}\big||y|-{\mu_k} \big|^{l+\sigma} \, U_{x^{+}_{k,1,r}, \Lambda}^{m^*} \Big)+O  \Big(\Big(\frac{k}{{\mu_k}}\Big)^{N-\epsilon_0}\Big)
					\\[2mm]
					&=  \frac{\mathcal {C}(r, \Lambda)}{{\mu_k}^{l+\sigma}}
					\,+\,O  \Big(\Big(\frac{k}{{\mu_k}}\Big)^{N-\epsilon_0}\Big).
				\end{align*}

				Then, we can obtain that
				\begin{align}\label{estimateofI21}
					I_{21}
					= &
					\int_{\R^N}|U_{0,1}|^{m^*}-\frac{c_0}{\Lambda^{l}{\mu_k}^{l}}
					\int_{\R^N}|y_1|^{l}U_{0,1}^{m^*}
					\nonumber\\[2mm]
					&-\frac12 l(l-1)\frac{c_0}{\Lambda^{l-2}{\mu_k}^{l}}
					\int_{\R^N}|y_1|^{l-2}U_{0,1}^{m^*}
					({\mu_k} -r)^2
					\nonumber\\[2mm]
					&+\frac{\mathcal {C}(r, \Lambda)}{{\mu_k}^{l}}({\mu_k} -r)^{3}
					+\frac{\mathcal {C}(r, \Lambda)}{{\mu_k}^{l+\sigma}}
					+O\Big(\Big(\frac{k}{{\mu_k}}\Big)^{N-\epsilon_0}\Big).
				\end{align}

				Finally, we consider  $I_{22} $
	\begin{equation}\label{I22}
	    \begin{aligned}
					I_{22}
					\,=\,&m^{*} \int_{\Omega_1^{+}}  U_{x^{+}_{k,1,r}, \Lambda}^{m^*-1}  \Big(\sum_{j=2}^k  U_{x^{+}_{k,j,r}, \Lambda}+\sum_{j=1}^k U_{x^{-}_{k,j,r}, \Lambda}\Big)\\
					&+m^{*}\int_{\Omega_1^{+}} \Big[K\Big(\frac{|y|}{{\mu_k}}\Big) -1\Big]  U_{x^{+}_{k,1,r}, \Lambda}^{m^*-1}  \Big(\sum_{j=2}^k  U_{x^{+}_{k,j,r}, \Lambda}+\sum_{j=1}^k U_{x^{-}_{k,j,r}, \Lambda}\Big)\\
					=&k \Big(1-\frac{2}{m^*}\Big) \int_{\R^N}|U_{0,1}|^{m^*}-k \int_{\R^N}  U_{x^{+}_{k,1,r}, \Lambda}^{m^*-1}  \Big(\sum_{j=2}^k  U_{x^{+}_{k,j,r}, \Lambda}+\sum_{j=1}^k U_{x^{-}_{k,j,r}, \Lambda}\Big)\\
     &+O\Big(\Big(\frac{k}{{\mu_k}}\Big)^{N-\epsilon_0}\Big)+O \Big(\frac{1}{k^{l}} \Big(\frac{k}{{\mu_k}}\Big)^{N-2m}\Big).\\
				\end{aligned}
	\end{equation}

				Combining \eqref{estimateofI1}, \eqref{estimateofI2}, \eqref{estimateofI21}, we can get
				\begin{align*}
					I(W_{r,h,\Lambda})
					\,=\,&k \Big(1-\frac{2}{m^*}\Big) \int_{\R^N}|U_{0,1}|^{m^*}-k \int_{\R^N}  U_{x^{+}_{k,1,r}, \Lambda}^{m^*-1}  \Big(\sum_{j=2}^k  U_{x^{+}_{k,j,r}, \Lambda}+\sum_{j=1}^k U_{x^{-}_{k,j,r}, \Lambda}\Big)
					\\[2mm]
					&+\frac{2k}{m^*} \Big[\frac{c_0}{\Lambda^{l}{\mu_k}^{l}}
					\int_{\R^N}|y_1|^{l}U_{0,1}^{m^*}+\frac{c_0 l(l-1)}{2 \Lambda^{l-2}{\mu_k}^{l}}
					\int_{\R^N}|y_1|^{l-2}U_{0,1}^{m^*}({\mu_k}-r)^2\Big]
					\\[2mm]
					&+k\frac{\mathcal {C}(r, \Lambda)}{{\mu_k}^{l}}({\mu_k} -r)^{3}
					+k\frac{\mathcal {C}(r, \Lambda)}{{\mu_k}^{l+\sigma}}+k O\Big(\Big(\frac{k}{{\mu_k}}\Big)^{N-\epsilon_0}\Big)+k O\Big(\frac{1}{k^{l}} \Big(\frac{k}{{\mu_k}}\Big)^{N-2m}\Big).
				\end{align*}		
			\end{proof}

			\medskip
			Combining Lemmas \ref{express7}-\ref{express9}, we can get
			the following Proposition which gives the expression of  $ I(W_{r,h,\Lambda}).  $
			\medskip
			
			\begin{proposition} \label{func}
				Suppose that  $ K(|y|)  $ satisfies  ${(\bf K)} $ and  $N\ge 2m+3 $, $(r,h,\Lambda) \in{{\mathscr D}_k}$. Then we have
    \begin{align} \label{express5}
					I(W_{r,h,\Lambda})
					\,=\,& kA_1 -\frac{k}{\Lambda^{N-2m}}\Big[\,\frac{B_4 k^{N-2m}}{(r \sqrt{1-h^2})^{N-2m}}\,+\, \frac{B_5 k}{r^{N-2m} h^{N-2m-1} \sqrt {1-h^2}}\,\Big]
					\nonumber\\[2mm]
					&+k\Big[\frac{A_2}{\Lambda^{l}  k^{\frac{l(N-2m)}{N-2m-l}}}
					+\frac{A_3}{\Lambda^{l-2}  k^{\frac{l(N-2m)}{N-2m-l}}}({\mu_k}-r)^2\Big]
					+k\frac{\mathcal {C}(r, \Lambda)}{k^{\frac{l(N-2m)}{N-2m-l}}}({\mu_k} -r)^{3}
					\nonumber\\[2mm]
					&+k\frac{\mathcal {C}(r, \Lambda)}{k^{\frac{l(N-2m)}{N-2m-l}+\sigma}}
					+kO\Big(\frac{1}{k^{\big(\frac{l(N-2m)}{N-2m-l}+\frac{2(N-2m-1)}{N-2m+1}+\sigma\big)}}\Big),
				\end{align}
			as $k \to \infty$, where $ A_i,(i=1,2,3), B_4, B_5  $ are positive constants.
		\end{proposition}

		\begin{proof}
			A direct result of Lemmas \ref{express7}-\ref{express9} is
			\begin{align*}
				I(W_{r,h,\Lambda})
				\,=\,& k A_1 \,-\,\frac{k}{\Lambda^{N-2m}} \Big[\,\frac{B_4k^{N-2m}}{(r \sqrt{1-h^2})^{N-2m}}\,+\,\frac{B_5k}{r^{N-2m} h^{N-2m-1} \sqrt {1-h^2}}\,\Big]
				\nonumber\\[2mm]
				&\,+\,k\Big[\frac{A_2}{\Lambda^{l}{\mu_k}^{l}}\,+\,\frac{A_3}{\Lambda^{l-2}{\mu_k}^{l}}({\mu_k}-r)^2\Big]
				\,+\,k\frac{\mathcal {C}(r,\Lambda)}{k^{\frac{l(N-2m)}{N-2m-l}}}({\mu_k} -r)^{3}
				\\[2mm]
				&\,+\,k\frac{\mathcal {C}(r,\Lambda)}{k^{\frac{(N-2m)l}{N-2m-l}+\sigma}}
				\,+\,kO\Big(\Big(\frac{k}{{\mu_k}}\Big)^{N-\epsilon_0}\Big)
				\,+\,kO\Big(\frac{1}{k^{l}} \Big(\frac{k}{{\mu_k}}\Big)^{N-2m}\Big)
				\nonumber \\[2mm]
				&\,+\,kO\Big(\frac{\zeta_1(k) k^{N-2m}}{\big(r \sqrt{1-h^2}\big)^{N-2m}}\Big)
				+k O\Big(\frac{ \zeta_2(k)k}{r^{N-2m} h^{N-2m-1} \sqrt {1-h^2}}\Big),
			\end{align*}
			with $B_4=B_0B_1, B_5=B_0B_2$ are positive constants. From the expressions of  $ \zeta_1(k), \zeta_2(k)  $ and asymptotic expression  of  $h, r $ as in \eqref{definitionofsk},\eqref{sigmak}  $,  $ we can show that
			\[ \frac{\zeta_1(k) k^{N-2m}}{\big(r \sqrt{1-h^2}\big)^{N-2m}},
			\quad \frac{\zeta_2(k)k}{r^{N-2m} h^{N-2m-1} \sqrt {1-h^2}} , \]
			can be absorbed in  $ O\Big(\frac{1}{k^{\big(\frac{l(N-2m)}{N-2m-l}+\frac{2(N-2m-1)}{N-2m+1}+\sigma\big)}}\Big). $
			
			\medskip
			Noting that
			$ l > l_1 $ implies
			\[  \frac{N-2m-1}{N-2m+1} < \frac{l}{N-2m-1},  \]
			thus provided with  $ \epsilon_0, \sigma$ small enough, we can get
			\[\Big(\frac{k}{{\mu_k}}\Big)^{N-2m+2-\epsilon_0}
			= \frac 1 {k^{\frac{l(2-\epsilon_0)}{N-2m-l}}} \frac 1 {k^{\frac{l(N-2m)}{N-2m-l}}}
			\le C \frac{1}{k^{\big(\frac{l(N-2m)}{N-2m-l}+\frac{2(N-2m-1)}{N-2m+1}+\sigma\big)}}. \]
			Since  $ l\ge 2 $, we can check that
			\[ \frac{1}{k^{l}} \Big(\frac{k}{{\mu_k}}\Big)^{N-2m}  \le \frac{C}{k^{\big(\frac{l(N-2m)}{N-2m-l}+\frac{2(N-2m-1)}{N-2m+1}+\sigma\big)}}.\]
			Thus we can get \eqref{express5}.
		\end{proof}

		\medskip
		From Proposition \ref{pro2.4},  in order to show the existence of critical point of $I(u)$, we would   find a critical point of  function $F(r,h,\Lambda):= I( W_{r, h, \Lambda}+\phi_{r,h, \Lambda}) $.
		To get the expansions of $\frac{F(r,h,\Lambda)}{\partial \Lambda}, \frac{F(r,h,\Lambda)}{\partial h}$, we need the following expansions  for $ \frac{\partial I(W_{r,h,\Lambda})}{\partial \Lambda},  \frac{\partial I(W_{r,h,\Lambda})}{\partial h}$.
  \begin{proposition}\label{func2}
			Suppose that  $ K(|y|)  $ satisfies   ${(\bf K)} $ and  $N\ge 2m+3 $, $(r,h,\Lambda) \in{{\mathscr D}_k}$.
			We have
			\begin{align}\label{express55}
				\frac{\partial I(W_{r,h,\Lambda})}{\partial \Lambda}
				\,=\,&\frac{k (N-2m)}{\Lambda^{N-2m+1}} \Big[\frac{B_4 k^{N-2m}}{(r \sqrt{1-h^2})^{N-2m}}+\frac{B_5 k}{r^{N-2m} h^{N-2m-1} \sqrt {1-h^2}}\,\Big] \nonumber
				\\[2mm]
				&-k \Big[\,\frac{l A_2}{\Lambda^{l+1}k^{\frac{(N-2m)l}{N-2m-l}}}+\frac{(l-2) A_3}{\Lambda^{l-1} k^{\frac{(N-2m)l}{N-2m-l}}} ({\mu_k}-r)^2\,\Big]
				+kO \Big(\frac 1 {k^{\frac{(N-2m)l}{N-2m-l}+\sigma}}\Big),
			\end{align}
   and
   \begin{align} \label{express6}
				\frac{\partial I(W_{r,h,\Lambda})}{\partial h}
				\,=\,&  -\frac{k}{\Lambda^{N-2m}}\Big[\, (N-2m) \frac{B_4 k^{N-2m}}{ r^{N-2m}( \sqrt{1-h^2})^{N-2m+2}} h -(N-2m-1) \frac{B_5 k}{ r^{N-2m}h^{N-2m} \sqrt {1-h^2}}  \,\Big]
				\nonumber\\[2mm]
				& \quad  +kO\Big(\frac{1}{k^{\big(\frac{l(N-2m)}{N-2m-l}+\frac{(N-2m-1)}{N-2m+1}+M_{m,N,l}-\epsilon_0\big)}}\Big),
			\end{align}
			as $k \to \infty$, where the constants  $B_i, i= 4,5 $ and  $A_i, i=2,3 $ are  defined in Proposition \ref{func}, and $\epsilon_0$ is a small fix constant.
			
		\end{proposition}
		
		\begin{proof} The proof of \eqref{express55} is similar to the proof of Proposition 4.4 in \cite{Wei2}. For brevity, We omit it here and focus on the proof of \eqref{express6}.
			Recall that
			\begin{align} \label{estimate kernel}
				\overline{\mathbb{Z}}_{2j}
				\le  C  \frac{r} {(1+|y-x^{+}_{k,j,r}|  )^{N-2m+1}  },
				\quad   \underline{\mathbb{Z}}_{2j}  \le  C  \frac{r} {(1+|y-x^{-}_{k,j,r}|  )^{N-2m+1}  }.
			\end{align}
			We know that

\begin{eqnarray*}
		\frac{\partial I(W_{r,h,\Lambda})}{\partial h}=   \left\{\begin{array}{rcl}\displaystyle\frac{1}{2}  \frac{\partial } {\partial h } \displaystyle\int_{\mathbb{R}^N}|\nabla W_{r,h,\Lambda}|^2- \frac{1}{m^*}  \frac{\partial } {\partial h}   \displaystyle\int_{\mathbb{R}^N} K\Big(\frac{|y|}{{\mu_k}}\Big) W_{r,h,\Lambda}^{m^*},\qquad \text{if }m \text{ is even},
			&&\\[2mm]
   \space\\[2mm]
		\displaystyle \frac{1}{2}  \frac{\partial } {\partial h } \displaystyle\int_{\mathbb{R}^N}|\nabla W_{r,h,\Lambda}|^2- \frac{1}{m^*}  \frac{\partial } {\partial h}   \displaystyle\int_{\mathbb{R}^N} K\Big(\frac{|y|}{{\mu_k}}\Big) W_{r,h,\Lambda}^{m^*}.\qquad \text{if } m \text{ is odd}. \end{array}\right.		
	\end{eqnarray*}
Then,
			\begin{align}\label{express10}
				\frac{\partial I(W_{r,h,\Lambda})}{\partial h}
				\,=\,&  k \frac{\partial } {\partial h } \int_{\mathbb{R}^N} U_{x^{+}_{k,1,r}, \Lambda}^{m^*-1}\Big(\sum_{j=2}^kU_{x^{+}_{k,j,r}, \Lambda}\,+\,\sum_{i=1}^k  U_{x^{-}_{k,j,r}, \Lambda}\Big)
				\nonumber\\[2mm]
				\quad &- \int_{\mathbb{R}^N} K\Big(\frac{|y|}{{\mu_k}}\Big) W_{r,h,\Lambda}^{m^*-1}    \Big(   \overline{\mathbb{Z}}_{21}
				+ \sum_{j=2 }^k  \overline{\mathbb{Z}}_{2j}
				+ \sum_{j=1}^k   \underline{\mathbb{Z}}_{2j}   \Big).
			\end{align}
			From \eqref{express10}, similar to the  calculations in the proof of Proposition \ref{express9},  we can get
			\begin{align}   \label{frac h}
				\frac{\partial I(W_{r,h,\Lambda})}{\partial h} =  - k    \frac{\partial }{\partial h}  \int_{\R^N}   U_{x^{+}_{k,1,r}, \Lambda}^{m^*-1}   \Big(\sum_{j=2}^kU_{x^{+}_{k,j,r}, \Lambda}\,+\,\sum_{i=1}^k  U_{x^{-}_{k,j,r}, \Lambda}\Big)  +  k^2O  \Big(\Big(\frac{k}{{\mu_k}}\Big)^{N-\epsilon_0}\Big).
			\end{align}
			
			Then by some  tedious but straightforward analysis, we can get
			\begin{align}
				\frac{\partial I(W_{r,h,\Lambda})}{\partial h}
				\,=\,&  -\frac{k}{\Lambda^{N-2m}}\Big[\, (N-2m) \frac{B_4 k^{N-2m}}{ r^{N-2m}( \sqrt{1-h^2})^{N-2m+2}} h -(N-2m-1) \frac{B_5 k}{ r^{N-2m}h^{N-2m} \sqrt {1-h^2}}  \,\Big]	\nonumber\\[2mm]
				& \quad \quad +h \frac{B_5 k}{r^{N-2m} h^{N-2m-1}({1-h^2})^{\frac 32}}  \,\Big]  +   k^2O  \Big(\Big(\frac{k}{{\mu_k}}\Big)^{N-2m+2-\epsilon_0}\Big),
			\end{align}
			for some $\epsilon_0$ small enough.
			In fact, we  know that  $ k  \Big(\frac{k}{{\mu_k}}\Big)^{N-2m+2-\epsilon_0}$ and $  h \frac{B_5 k}{r^{N-2m} h^{N-2m-1}({1-h^2})^{\frac 32}} $  can be absorbed in   $O\Big(\frac{1}{k^{\big(\frac{l(N-2m)}{N-2m-l}+\frac{N-2m-1}{N-2m+1}+M_{m,N,l}-\epsilon_0\big)}}\Big)$, providing $ l $ satisfying \eqref{assumptionform}, $\epsilon_1$  and $ \epsilon_0 $ small enough.   Then  we can get  \eqref{express6} directly.
		\end{proof}
		
		\medskip
        Before we find a critical point of $F(r,h,\Lambda)$ in $D_k$, we clear and accurate expansions of $\frac{F(r,h,\Lambda)}{\partial \Lambda}$ and $ \frac{F(r,h,\Lambda)}{\partial h}$. Recall that
        \begin{equation*}
            {\lambda_k} = \frac{h_0}{k^{\frac{N-2m-1}{N-2m+1}}}, \quad \text{with}   ~h_0=  \Big[\frac{(N-2m-1) B_5}{(N-2m) B_4}\Big]^{\frac 1 {N-2m+1}},
        \end{equation*}
        we can first rewrite the expansion of energy functional at the point $\lambda_k$.
\medskip
	\begin{proposition} \label{A5}
				Suppose that  $ K(|y|)  $ satisfies  ${(\bf K)} $ and  $N\ge 2m+3 $, $(r,h,\Lambda) \in{{D}_k},\lambda_k$ is defined in \eqref{h_0}, and $D_k$ is defined in \eqref{Dk}. Then we have
				\begin{align}\label{FF}
			F(r,h,\Lambda)\,=\, &k  A_1 -k \Big[\frac{B_4}{\Lambda^{N-2m} k^{\frac{(N-2m)l}{N-2m-l}}}+\frac{B_6}{\Lambda^{N-2m} k^{\frac{(N-2m)l}{N-2m-l}+\frac{2(N-2m-1)}{N-2m+1}}}
			\nonumber\\[2mm]
			&+\frac{B_7}{\Lambda^{N-2m}  k^{\frac{(N-2m)l}{N-2m-l}+\frac{2(N-2m-1)}{N-2m+1}}} (1-{\lambda_k}^{-1} h)^2\Big]
			\nonumber\\[2mm]
			&+k \Big[\frac{A_2}{\Lambda^{l}  k^{\frac{(N-2m)l}{N-2m-l}}}+\frac{A_3}{\Lambda^{l-2}  k^{\frac{(N-2m)l}{N-2m-l}}}({\mu_k}-r)^2\Big] +k\frac{\mathcal {C}(r, \Lambda)}{k^{\frac{(N-2m)l}{N-2m-l}}} ({\mu_k}-r)^{3}  \nonumber
			\\[2mm]
			&+k O \Big(\frac{1}{k^{\frac{(N-2m)l}{N-2m-l}+\frac{2(N-2m-1)}{N-2m+1}}}\Big) (1-{\lambda_k}^{-1} h)^3+ k O\Big(\frac{1}{k^{\big(\frac{l(N-2m)}{N-2m-l}+\frac{2(N-2m-1)}{N-2m+1}+M_{m,N,l}-\epsilon_0\big)}}\Big),
		\end{align}
		\begin{align}\label{frach0}
			\frac{\partial F(r,h,\Lambda)}{\partial \Lambda}
			\,=\,&k  \Big[\frac{ (N-2m) B_4}{\Lambda^{N-2m+1} k^{\frac{(N-2m)l}{N-2m-l}}} - \frac{lA_2}{\Lambda^{l+1} k^{\frac{(N-2m)l}{N-2m-l}}} \Big]
			\nonumber\\[2mm]
			&+\frac{(l-2)A_3}{\Lambda^{l-1}  k^{\frac{(N-2m)l}{N-2m-l}}} ({\mu_k}-r)^2
			+kO \, \Big(\frac1{k^{\frac{(N-2m)l}{N-2m-l}}} ({\mu_k}-r)^{3}\Big)
		\end{align}
		and
		\begin{align} \label{frach1}
			\frac{\partial F(r,h, \Lambda)}{\partial h }
			&\,=\,     \frac{k}{\Lambda^{N-2m}}\Big[\,  \frac{2B_7}{\Lambda^{N-2m}  k^{\frac{(N-2m)l}{N-2m-l}+\frac{(N-2m-1)}{N-2m+1}}} (1-{\lambda_k}^{-1} h)  \,\Big]     \nonumber\\[2mm]
			& \quad
			+kO\Big(\frac{1}{k^{\big(\frac{l(N-2m)}{N-2m-l}+\frac{N-2m-1}{N-2m+1} \big)}}\Big) (1-{\lambda_k}^{-1} h)^2 +kO\Big(\frac{1}{k^{\big(\frac{l(N-2m)}{N-2m-l}+\frac{N-2m-1}{N-2m+1}+\sigma\big)}}\Big).
		\end{align}
			as $k \to \infty$, where the constants  $B_i, i= 4,5 $ and  $A_i, i=2,3 $ are  defined in Proposition \ref{func}, and
   \begin{equation}\label{B67}
       B_6:= \frac{(N-2m) B_4 {h_0}^2} 2+\frac{B_5}{{h_0}^{N-2m-1}}  , \quad  B_7:= \frac{(N-2m)} 2 \Big[B_4 {h_0}^2 +\frac{(N-2m-1)  B_5}{{h_0}^{N-2m-1}} \big ].
   \end{equation}
		\end{proposition}
\begin{proof}		

  Denote
  		\begin{align*}
			\mathcal {F}(h) &:=
			\frac{B_4 k^{N-2m}}{(\sqrt{1-h^2})^{N-2m}}\,+\, \frac{B_5 k}{h^{N-2m-1} \sqrt {1-h^2}},
		\end{align*}
		then
		\begin{align*}
			\mathcal {F} '(h)
			&=\Big[(N-2m)  B_4 k^{N-2m} h-(N-2m-1) \frac{B_5 k}{h^{N-2m}}\Big]+O\Big(\frac{k}{h^{N-2m-2}}\Big),
		\end{align*}
		and
		\begin{align}
			\mathcal {F} ''(h)
			\,=\,&(N-2m)  B_4 k^{N-2m}
			+(N-2m-1)(N-2m)\frac{B_5 k}{h^{N-2m+1}}\nonumber\\[2mm]
   &\qquad\qquad\qquad\qquad\qquad+O\big(h^2  k^{N-2m}\big)+O\Big(\frac{k}{h^{N-2m-1}}\Big),
		\end{align}
		\begin{equation}
			\mathcal {F} '''(h)\,=\,  O \Big(\frac{k}{h^{N-2m+2}}\Big).
		\end{equation}
		
		Expanse $\mathcal {F}(h)$ at the point $\lambda_k$, we have
		\begin{align}  \label{FFF1}
			\mathcal {F}(h)= &\mathcal {F}({\lambda_k})+\mathcal {F} '({\lambda_k}) (h-{\lambda_k})+  \frac{1}{2}  \mathcal {F} '' ({\lambda_k}) (h-{\lambda_k})^2+O\Big(\mathcal {F} ''' \big({\lambda_k}+(1-\nu)  h \big)\Big)(h-{\lambda_k})^3,
		\end{align}
		where
		\begin{align}\label{FFF2}
			\mathcal {F}({\lambda_k}) = B_4 k^{N-2m}  \Big[1+\frac{N-2m}{2} {\lambda_k}^2+O({\lambda_k}^4)\Big]
			+\frac{B_5 k}{{\lambda_k}^{N-2m-1}}  \Big[1+\frac{1} 2{\lambda_k}^2+O({\lambda_k}^4)\Big],
		\end{align}
  and
		\begin{align}\label{FFF3}
			\mathcal {F}'({\lambda_k})  = O\Big(\frac{k}{{\lambda_k}^{N-2m-2}}\Big)
		\end{align}
		\begin{align}\label{FFF4}
			\mathcal {F} '' ({\lambda_k}) =  \frac{(N-2m)} 2 \Big[B_4 k^{N-2m}+(N-2m-1)  \frac{B_5 k}{{\lambda_k}^{N-2m+1}} \Big] + O\big({\lambda_k}^2  k^{N-2m}\big).
		\end{align}
		Then combining \eqref{FFF1}-\eqref{FFF4},  we can get
		\begin{align*}
			\mathcal {F}(h)
			=&B_4 k^{N-2m}  +  \frac{(N-2m)B_4}{2}  k^{N-2m}{\lambda_k}^2  +\frac{B_5 k}{{\lambda_k}^{N-2m-1}}  + O\Big(\frac{k}{{\lambda_k}^{N-2m-2}}\Big) (h-{\lambda_k})  \nonumber\\[2mm]
			&
			+\frac{(N-2m)} 2 \Big[B_4 k^{N-2m}+(N-2m-1)  \frac{B_5 k}{{\lambda_k}^{N-2m+1}} \Big](h-{\lambda_k})^2
			+O\Big(\frac{k}{{\lambda_k}^{N-2m+2}}\Big) (h-{\lambda_k})^3\nonumber\\[2mm]
   =&B_4  k^{N-2m}+B_6  \frac{k^{N-2m}}{k^{\frac{2(N-2m-1)}{N-2m+1}}}
			+B_7 \frac{k^{N-2m}}{k^{\frac{2(N-2m-1)}{N-2m+1}}}
			(1-{\lambda_k}^{-1} h)^2 +O \Big(\frac{k^{N-2m}}{k^{\frac{2(N-2m-1)}{N-2m+1}}}\Big) (1-{\lambda_k}^{-1} h)^3.
		\end{align*}
		where $B_6$ and $B_7$ are defined in \eqref{B67}.
		
		\medskip
		Since $(r,h,\Lambda) \in{{D}_k}$, which indicates that
		\[r\in \Big[k^{\frac{N-2m}{N-2m-l}}-\frac{1}{k^{\bar \theta}}, \quad k^{\frac{N-2m}{N-2m-l}}+\frac{1}{k^{\bar \theta}}\Big],\]
		then
		\begin{align*}
			r^{N-2m}= k^{\frac{(N-2m)^2}{N-2m-l}} \Big(1+\frac{\mathcal {C}(r, \Lambda)}{{k^{\frac{(N-2m)}{N-2m-l}+\bar \theta}}}\Big).
		\end{align*}
		Thus,
		\begin{align*}
			& \frac{B_4 k^{N-2m}}{(r \sqrt{1-h^2})^{N-2m}}\,+\,\frac{B_5 k}{r^{N-2m} h^{N-2m-1}\sqrt {1-h^2}}
			\nonumber\\[2mm]
			&=\frac{B_4}{k^{\frac{(N-2m)l}{N-2m-l}}}+\frac{B_6}{k^{\frac{(N-2m)l}{N-2m-l}+\frac{2(N-2m-1)}{N-2m+1}}}+\frac{\mathcal {C}(r, \Lambda)}{k^{\frac{(N-2m)l}{N-2m-l}+M_{m,N,l}-\epsilon_0}}
			\nonumber\\[2mm]
			&\quad+\frac{B_7}{k^{\frac{(N-2m)l}{N-2m-l}+\frac{2(N-2m-1)}{N-2m+1}}} (1-{\lambda_k}^{-1} h)^2+O \Big(\frac{1}{k^{\frac{(N-2m)l}{N-2m-l}+\frac{2(N-2m-1)}{N-2m+1}}}\Big) (1-{\lambda_k}^{-1} h)^3.
		\end{align*}
		According to the expansion above, we can rewrite $ F(r,h, \Lambda)$ as
		\begin{align*}
			F(r,h,\Lambda)\,=\, &k  A_1 -k \Big[\frac{B_4}{\Lambda^{N-2m} k^{\frac{(N-2m)l}{N-2m-l}}}+\frac{B_6}{\Lambda^{N-2m} k^{\frac{(N-2m)l}{N-2m-l}+\frac{2(N-2m-1)}{N-2m+1}}}
			\nonumber\\[2mm]
			&+\frac{B_7}{\Lambda^{N-2m}  k^{\frac{(N-2m)l}{N-2m-l}+\frac{2(N-2m-1)}{N-2m+1}}} (1-{\lambda_k}^{-1} h)^2\Big]
			\nonumber\\[2mm]
			&+k \Big[\frac{A_2}{\Lambda^{l}  k^{\frac{(N-2m)l}{N-2m-l}}}+\frac{A_3}{\Lambda^{l-2}  k^{\frac{(N-2m)l}{N-2m-l}}}({\mu_k}-r)^2\Big] +k\frac{\mathcal {C}(r, \Lambda)}{k^{\frac{(N-2m)l}{N-2m-l}}} ({\mu_k}-r)^{3}  \nonumber
			\\[2mm]
			&+k O \Big(\frac{1}{k^{\frac{(N-2m)l}{N-2m-l}+\frac{2(N-2m-1)}{N-2m+1}}}\Big) (1-{\lambda_k}^{-1} h)^3+ k O\Big(\frac{1}{k^{\big(\frac{l(N-2m)}{N-2m-l}+\frac{2(N-2m-1)}{N-2m+1}+M_{m,N,l}-\epsilon_0\big)}}\Big),
		\end{align*}
		then we get \eqref{FF}.
		Similarly, we can easily get \eqref{frach0} and \eqref{frach1}.

  \end{proof}

		\section{Basic estimates and lemmas} \label{appendixB}
		\begin{lemma} \label{b.0}
			Suppose that  $(r,h,\Lambda) \in {{\mathscr D}_k} $, for  $ y\in \Omega_1^{+}  $  there exists a constant $C$ such that
			\begin{align}
				\Big(\sum_{j=2}^k  U_{x^{+}_{k,j,r}, \Lambda}+\sum_{j=1}^k U_{x^{-}_{k,j,r}, \Lambda}\Big)
				\le  C  \frac{1}{\big(1+|y-x^{+}_{k,1,r}|\big)^{N-2m-\alpha_1}}\Big(\frac{k}{{\mu_k}}\Big)^{\alpha_1},
			\end{align}
   and
			\begin{align}
				\Big(  \sum_{j=2 }^k  \overline{\mathbb{Z}}_{2j}  +   \sum_{j=1 }^k   \underline{\mathbb{Z}}_{2j}   \Big)
				\le  C  \frac{\mu_k}{\big(1+|y-x^{+}_{k,1,r}|\big)^{N-2m+1-\alpha_2}}\Big(\frac{k}{{\mu_k}}\Big)^{\alpha_2},
			\end{align}
			with  $ \alpha_1=(1, N-2m), \alpha_2=(1, N-2m+1)  $.
		\end{lemma}
		\begin{proof}
 For  $ y\in \Omega_1^{+}  $ and $ j= 2, \cdots, k  $, we can easily obtain that
    \begin{equation*}
    |y-x^{+}_{k,j,r}|\ge\frac 14|x^{+}_{k,1,r}-x^{+}_{k,j,r}|,\qquad |y-x^{-}_{k,j,r}| \ge \frac 14|x^{+}_{k,1,r}-x^{-}_{k,1,r}|\ge C\Big(\frac {r} k\Big).
    \end{equation*}

Then
\begin{align*}
 \Big(\sum_{j=2}^k  U_{x^{+}_{k,j,r}, \Lambda}&+\sum_{j=1}^k U_{x^{-}_{k,j,r}, \Lambda}\Big)\\
 &\le\frac{C}{\big(1+|y-x^{+}_{k,1,r}|\big)^{N-2m-\alpha_1}} \Big[\sum_{j=2}^k  \frac{1}{\big(1+|y-x^{+}_{k,j,r}|\big)^{\alpha_1}}+\frac{1}{\big(1+|y-x^{-}_{k,1,r}|\big)^{\alpha_1}}\Big]
\\[2mm]
&\le \frac{C}{\big(1+|y-x^{+}_{k,1,r}|\big)^{N-2m-\alpha_1}}  \Big[\sum_{j=2}^k \frac{1}{|x^{+}_{k,1,r}-x^{+}_{k,j,r}|^{\alpha_1}}+\frac{1}{|x^{+}_{k,1,r}-x^{-}_{k,1,r}|^{\alpha_1}}\Big]
\\[2mm]
&\le \frac{C}{\big(1+|y-x^{+}_{k,1,r}|\big)^{N-2m-\alpha_1}} \Big(\frac{k}{\mu_k}\Big)^{\alpha_1} .
\end{align*}
   Similarly,
   \begin{align*}
 \Big(  \sum_{j=2 }^k  \overline{\mathbb{Z}}_{2j}  &+   \sum_{j=1 }^k   \underline{\mathbb{Z}}_{2j}   \Big)\\
 &\le\frac{Cr}{\big(1+|y-x^{+}_{k,1,r}|\big)^{N-2m+1-\alpha_2}} \Big[\sum_{j=2}^k  \frac{1}{\big(1+|y-x^{+}_{k,j,r}|\big)^{\alpha_2}}+\frac{1}{\big(1+|y-x^{-}_{k,1,r}|\big)^{\alpha_2}}\Big]
\\[2mm]
&\le C  \frac{\mu_k}{\big(1+|y-x^{+}_{k,1,r}|\big)^{N-2m+1-\alpha_2}}\Big(\frac{k}{{\mu_k}}\Big)^{\alpha_2} .
\end{align*}

		\end{proof}
		\medskip

			\begin{lemma}   \label{B22}
			Suppose that  $N\ge 2m+3 $ and $m$ satisfies \eqref{assumptionform}.  We have
			\begin{align}  \label{eee}
				\frac{\mu_k}{k^{(\frac l{N-2m-l})(
						N+2m-2\frac{N-2m-l}{N-2m}- 2\epsilon_1 )}} \le     \frac{C}{k^{\big(\frac{l(N-2m)}{N-2m-l}+\frac{N-2m-1}{N-2m+1}+M_{m,N,l}-\epsilon_0\big)}},
			\end{align}
			provided with   $ \epsilon_0$ and $ \epsilon_1$ small enough.
		\end{lemma}
		\begin{proof}     It's easy to show that
			\begin{align*}
				\frac{\mu_k}{k^{(\frac {2l(N-2m)}{N-2m-l}) }} \le     \frac{C}{k^{\big(\frac{l(N-2m)}{N-2m-l}+\frac{N-2m-1}{N-2m+1}+M_{m,N,l}-\epsilon_0\big)}},
			\end{align*}
			for $l$ satisfying \eqref{assumptionform}.  In order to get \eqref{eee}, we just need to show			
			\begin{align}   \label{eeee}
				\frac {\mu_k} {k^{(\frac l{N-2m-l})(
						N+2m-2\frac{N-2m-l}{N-2m}- 2\epsilon_1 )}}  =    \frac{k^{\frac{N-2m}{N-2m-l} }}{k^{(\frac l{N-2m-l})(
						N+2m-2\frac{N-2m-l}{N-2m}- 2\epsilon_1 )}}\le     \frac{C}{k^{\big(\frac{l(N-2m)}{N-2m-l}+\frac{N-2m-1}{N-2m+1}+M_{m,N,l}-\epsilon_0\big)}},
			\end{align}      for some $\epsilon_0$ and $ \epsilon_1$ small. The problem to show \eqref{eeee} can be reduced to  show  that
   \begin{equation}\label{iii}
       6+   \frac{N-2m-1}{N-2m+1}  <  3\bigg(\frac{N-2m} {N-2m-l} \bigg) + 2\bigg(\frac{N-2m-l}{N-2m}\bigg).
   \end{equation}
   For $l$ satisfying  \eqref{assumptionform}, \eqref{iii} is naturally established. This fact concludes the proof.
		\end{proof}
  \medskip

		Assuming that $x_m$ and $x_n$ are two different fixed points in the set $\{x^{\xi}_{k,j,r}: j=1,\cdots,k\}$, where $\xi \in \{+,-\}$. We consider the following
		function
		\begin{equation*}
			g_{mn}(y)=\frac{1}{(1+|y-x_{n}|)^{\gamma_{1}}}\frac{1}{(1+|y-x_{m}|)^{\gamma_{2}}},\quad m \neq n,
		\end{equation*}
		where  $\gamma_{1}\geq 1 $ and  $\gamma_{2}\geq 1 $ are two constants.
		\begin{lemma}\label{lemb1}(Lemma B.1, \cite{Wei2})
			For any constants  $0<\upsilon\leq \min\{\gamma_{1},\gamma_{2}\} $, there is a constant  $C>0 $, such that
			$$
			g_{mn}(y)\leq \frac{C}{|x_{m}-x_{n}|^{\upsilon}}\Big(\frac{1}{(1+|y-x_{m}|)^{\gamma_{1}+\gamma_{2}-\upsilon}}+\frac{1}{(1+|y-x_{n}|)^{\gamma_{1}+\gamma_{2}-\upsilon}}\Big).
			$$
		\end{lemma}

		\begin{lemma}\label{lemb2}(Lemma 2.2, \cite{guo2})
			For any constant  $0<\beta <N-2m $, there is a constant  $C>0 $, such
			that
			$$
			\int_{\R^N} \frac{1}{|y-z|^{N-2m}}\frac{1}{(1+|z|)^{2m+\beta}}{\mathrm d}z \leq
			\frac{C}{(1+|y|)^{\beta}}.
			$$
		\end{lemma}
		
		Using Lemma \ref{lemb2}, we can have the following estimation:
		\begin{lemma}\label{laa3}
			
			Suppose that  $N\ge 2m+3 $ and  $\tau\in(0, 2), y=(y_1, \cdots, y_N)  $. Then there is a small
			$\sigma>0 $, such that
			when  $ y_3 \geq 0  $,
			\[
			\begin{split}
				&\int_{\R^N}\frac1{|y-z|^{N-2m}} W_{r,h,\Lambda}^{\frac{4m}{N-2m}}(z)\sum_{j=1}^k\frac1{(1+|z-x^{+}_{k,j,r}|)^{\frac{N-2m}{2}+\tau}}\, {\mathrm d}z \\[2mm]
				& \le C\sum_{j=1}^k\frac1{(1+|y-x^{+}_{k,j,r}|)^{\frac{N-2m}{2}+\tau+\sigma}},
			\end{split}
			\]
			and when  $  y_3 \le 0 $,
			\[
			\begin{split}
				&\int_{\R^N}\frac1{|y-z|^{N-2m}} W_{r,h,\Lambda}^{\frac{4m}{N-2m}}(z)\sum_{j=1}^k\frac1{(1+|z-x^{-}_{k,j,r}|)^{\frac{N-2m}{2}+\tau}}\, {\mathrm d}z \\[2mm]
				&\le C\sum_{j=1}^k\frac1{(1+|y-x^{-}_{k,j,r}|)^{\frac{N-2m}{2}+\tau+\sigma}}.
			\end{split}
			\]
			
		\end{lemma}
		
		\begin{proof}

   The proof of Lemma \ref{laa3} is similar to Lemma 2.3  in \cite{guo2}.  We omit it here.
		\end{proof}

		\medskip

\section*{Data availability}
No data was used for the research described in the article.

	\end{document}